\documentclass[11pt,reqno, a4paper]{amsart}
\usepackage{mathrsfs}
\usepackage{amsgen}
\usepackage{amscd}
\usepackage{booktabs}
\allowdisplaybreaks
\usepackage{amsmath}
\usepackage{latexsym}
\usepackage{amsfonts}
\usepackage{amssymb}
\usepackage{amsthm}
\usepackage{graphicx,epsfig}
\usepackage{color}
\usepackage{amsthm,amssymb}
\usepackage{graphicx}
\usepackage{sidecap}
\usepackage{enumerate}
\usepackage{amssymb,amsmath,graphicx,amsfonts,euscript}
\usepackage{color}
\usepackage{mathrsfs}
\usepackage{empheq}
\usepackage{cases}
\usepackage{cite}
\usepackage{ulem}
\usepackage{cancel}
\usepackage[margin=1.1in]{geometry}
\DeclareGraphicsExtensions{.eps}
\usepackage{amssymb,amsmath,graphicx,amsfonts,euscript,mathrsfs}
\usepackage{hyperref}

\usepackage{titlesec} 
\titleformat{\section}{\vskip10pt\large\bfseries}{\thesection.}{0.5em}{\centering\vspace{5pt}}
\titleformat{\subsection}{\vskip10pt\normalsize\bfseries}{\thesubsection.}{0.5em}{}

\newtheorem{theorem}{Theorem}[section]
\newtheorem{lemma}[theorem]{Lemma}

\newtheorem{remark}[theorem]{Remark}
\theoremstyle{definition}

\def\R{\mathbb{R}}
\def\Z{\mathbb{Z}}
\def\T{\mathbb{T}}
\def\C{\mathbb{C}}

\def\l{\langle}
\def\r{\rangle}

\newcommand{\fe}{\mathrm{e}}

\newcommand{\bZ}{{\mathbb Z}}

\numberwithin{equation}{section}

\begin{document}

\title[]{A fully discrete low-regularity integrator for the\\ 1D periodic cubic nonlinear Schr\"odinger equation}

\author[]{Buyang Li\,\,}
\address{\hspace*{-12pt}Buyang Li: 
Department of Applied Mathematics, The Hong Kong Polytechnic University,
Hung Hom, Hong Kong.}
\email{buyang.li@polyu.edu.hk}

\author[]{\,\,Yifei Wu}
\address{\hspace*{-12pt}Yifei Wu: Center for Applied Mathematics, Tianjin University, 300072, Tianjin, P. R. China.}
\email{yerfmath@gmail.com}




\maketitle

\begin{abstract}\noindent
A fully discrete and fully explicit low-regularity integrator is constructed for the one-dimensional periodic cubic nonlinear Schr\"odinger equation. The method can be implemented by using fast Fourier transform with $O(N\ln N)$ operations at every time level, and is proved to have an $L^2$-norm error bound of $O(\tau\sqrt{\ln(1/\tau)}+N^{-1})$ for $H^1$ initial data, without requiring any CFL condition, where $\tau$ and $N$ denote the temporal stepsize and the degree of freedoms in the spatial discretisation, respectively. 
\end{abstract}


\section{Introduction}\label{sec:introduction}

This article concerns the numerical solution of the cubic nonlinear Schr\"odinger (NLS) equation  
\begin{equation}\label{model}
 \left\{\begin{aligned}
& i\partial_tu(t,x)+\partial_{xx} u(t,x)
=\lambda|u(t,x)|^2u(t,x)
 &&\mbox{for}\,\,\, x\in\T\,\,\,\mbox{and}\,\,\, t\in(0,T] , \\
 &u(0,x)=u^0(x) &&\mbox{for}\,\,\,  x\in\T,
 \end{aligned}\right.
\end{equation}
on the one-dimensional torus $\T=(-\pi,\pi)$ with a nonsmooth initial value $u^0\in H^1(\T)$, where $\lambda=-1$ and $1$ are referred to as the focusing and defocusing cases, respectively. It is known that problem \eqref{model} is globally well-posed in $H^s(\T)$ for $s\ge 0$; see \cite{Bo}. 

The construction of numerical methods for the NLS equation and related dispersive equations with nonsmooth initial data has attracted much attention recently since the pioneering work of Ostermann \& Schratz \cite{Ostermann-Schratz-FoCM}, who introduced a low-regularity exponential-type integrator that could have first-order convergence in $H^{\gamma}(\T^d)$ for initial data $u^0\in H^{\gamma+1}(\T^d)$ and $\gamma>\frac{d}{2}$, 
where $d$ denotes the dimension of space. 
Before their work, the traditional regularity assumption for the NLS equation for a time-stepping method to have first-order convergence in $H^{\gamma}(\T^d)$ is $u^0\in H^{\gamma+2}(\T^d)$ for $\gamma\ge 0$ (losing two derivatives). This includes the Strang splitting methods \cite{ESS-2016,Lubich-2008}, the Lie splitting method \cite{Ignat-2011}, and classical exponential integrators \cite{Hochbruck-Ostermann-2010} (also see the discussion in \cite[p. 733]{Ostermann-Schratz-FoCM}). The finite difference methods \cite{Sanz-Serna1984,Wang2014} generally require more regularity of the initial data (one temporal derivative on the solution generally requires the initial data to have two spatial derivatives to satisfy certain compatibility conditions). 


The idea of Ostermann \& Schratz \cite{Ostermann-Schratz-FoCM} is to use twisted variable to reduce the consistency error in an exponential-type integrator, and to use harmonic analysis techniques to approximate the exponential integral. 
More recently, Wu \& Yao \cite{Wu-Yao-2020} applied different harmonic analysis techniques to construct a time-stepping method for the one-dimensional NLS equation with first-order convergence in $H^\gamma(\T)$ for initial data $u^0\in H^\gamma(\T)$ and $\gamma>\frac32$ (without losing any derivative). 
Ostermann, Rousset \& Schratz furthermore weakened the regularity assumption of initial data to $u^0\in H^1(\T)$ in \cite{Ostermann-Schratz-FoCM2} and $u^0\in H^s(\T)$ with $s\in(0,1]$ in \cite{Ostermann-Rousset-Schratz-JEMS} by using estimates in the discrete Bourgain spaces. For $u^0\in H^1(\T)$ these methods were proved to have $L^2$-norm error bounds of $O(\tau^{\frac{5}{6}})$ and $O(\tau^{\frac{7}{8}-\epsilon})$, respectively, for the one-dimensional NLS equation. 
A general framework of low-regularity integrators for nonlinear parabolic, dispersive and hyperbolic equations was introduced in \cite{Rousset-Schratz-2020}, where the condition for the numerical solution of the NLS equation to have first-order convergence in $L^2(\T)$ is $u^0\in H^{\frac54}(\T)$. 

%
%

Besides the NLS equation, the techniques of twisted variable and harmonic analysis techniques were also used in the construction of low-regularity integrators for other dispersive equations; see  \cite{Hofmanova-Schratz-2017, ostermann-su, diraclow, WuZhao-1, wu} and the references therein. 

As far as we know, the analysis of all the low-regularity integrators for the NLS equation are limited to semidiscretisation in time (the error from spatial discretisation is unknown for nonsmooth initial data), and the regularity condition for the time-stepping method to have first-order convergence is $u^0\in H^\gamma(\T)$ for $\gamma\ge\frac54$. 
We are only aware of a fully discrete Lawson-type exponential integrator for the Korteweg–de Vries equation \cite{ostermann-su}, with first-order convergence in $L^2(\T)$ in both time and space under a CFL condition $\tau=O(h)$ for solutions in $C(0, T; H^3(\T))$. 

The objective of this article is to construct a fully discrete and fully explicit lower-regularity integrator that has first-order convergence (up to a logarithmic factor) in both time and space for $H^1$ initial data. 
The temporal low-regularity integrator is constructed using twisted variables and with different harmonic analysis techniques in approximating the low- and high-frequency parts of the functions in the exponential integral. 
The spatial discretisation is integrated in the temporal low-regularity integrator by repeatedly using frequency truncation and Fast Fourier transform (FFT) techniques in every nonlinear operation (i.e., computing the product of two functions). By using a $(4N+1)$-point FFT for every product of two $(2N+1)$-term Fourier series in the numerical scheme and then truncating the obtained $(4N+1)$-term product series to $(2N+1)$-term again, we avoid generating trigonometric interpolation errors from using FFT. As a result, the spatial discretisation error of our method is purely due to frequency truncation and therefore can be analysed together with the temporal discretisation error in the frequency domain by using harmonic analysis techniques. 



The rest of this article is organised as follows. The fully discrete low-regularity integrator and the main theorem on the convergence rates of the method are presented in section \ref{section:result}. 
Some technical tools of harmonic analysis are presented in section \ref{section:tool}, which are used in section \ref{section:construction} in the construction of the numerical method and analysis of the consistency error. 
The error bound of proposed fully discrete low-regularity integrator is proved in section \ref{section:proof} by utilizing the consistency error bounds obtained in section \ref{section:construction} and the stability of the method, as well as the $H^1$-regularity of fully discrete numerical solution. The latter is proved to be bounded uniformly with respect to the temporal stepsize and the number of Fourier terms in the spatial discretisation. 
Numerical results are presented in section \ref{sec:numerical} to support the theoretical analysis in this article. 

\section{The numerical method and main theoretical result}
\label{section:result}

It is known that the solution of the NLS equation satisfies the following two conservation laws 
 (see e.g., \cite{Ca-book-03}):
\begin{enumerate}
\item
Mass conservation: 
\begin{align}
\frac1{2\pi}\int_\T |u(t,x)|^2\,d x = \frac1{2\pi}\int_\T |u^0(x)|^2\,d x  \quad\mbox{for}\,\,\, t>0 .  \label{mass}
\end{align}

\item
Momentum conservation: 
\begin{align}
\frac1{2\pi}\int_\T u(t,x)\partial_x\bar u(t,x)\,d x = \frac1{2\pi}\int_\T u^0\,\partial_x\bar u^0 \,d x   \quad\mbox{for}\,\,\, t>0 .  \label{momentum}
\end{align}
\end{enumerate}
These two conserved quantities will be approximated based on the initial data and utilized in the construction of the numerical method. 

We denote by $\Pi_0$ and $\Pi_{\ne 0}$ the zero-mode and nonzero-mode operators, respectively, defined by 
\begin{align}\label{def:ii}
\Pi_0 f =\frac1{2\pi} \int_\T f(x)\,dx
\quad\mbox{and}\quad 
\Pi_{\ne 0}(f)=\sum\limits_{k\in\Z, k\ne 0} e^{ i  kx } \hat{f}_k .
\end{align}
Then the conserved mass and momentum are denoted by 
\begin{align}\label{M+P}
M= \frac1{2\pi}\int_\T |u^0(x)|^2\,d x = \Pi_0(|u^0|^2) 
\quad\mbox{and} \quad
P =\frac1{2\pi} \int_{\T} u^0 \partial_x\overline{u^0}\,d x=\Pi_0\big(u^0\partial_x \overline{u^0}\big) , 
\end{align}
respectively.

For any positive integer $N$, we denote by $I_{2N}$ the $(4N+1)$-point trigonometric interpolation operator, which can be obtained through the discrete Fourier transform (see \cite{Chu-2008,DFT-wiki})
\begin{align}\label{def:iiN} 
I_{2N} f (x) =  \sum\limits_{k=-2N}^{2N} e^{i kx } \tilde{f}_k 
\quad\mbox{with}\quad
\tilde{f}_k  = \frac{1}{4N+1} \sum\limits_{n=-2N}^{2N} e^{- i k x_n }  f(x_n) 
\end{align} 
where 
$$
x_n = \frac{2\pi n}{ 4N +1 } \quad\mbox{for}\,\,\, n=-2N,\cdots, 2N . 
$$
If the Fourier coefficient $\hat f_k$ of the function $f$ satisfies that $\hat f_k=0$ for $|k|>2N$, then $I_{2N}f=f$ and therefore $\tilde f_k=\hat f_k$ in the formula \eqref{def:iiN}. In this case, both 
\begin{align}\label{def:iiN2} 
f(x_n) = \sum\limits_{k=-2N}^{2N} e^{i kx_n } \hat{f}_k , \quad n=-2N,\cdots,2N,
\end{align} 
and
\begin{align}\label{def:iiN3} 
\hat{f}_k  = \sum\limits_{n=-2N}^{2N} e^{- i k x_n }  f(x_n) \quad k=-2N,\cdots,2N,
\end{align} 
can be computed with cost $O(N\ln N)$ by using the fast Fourier transform (FFT); see \cite{Chu-2008}. 

Let $S_N$ be the subspace of functions $f\in L^2(\T)$ such that $\hat f_k=0$ for $|k|>N$. 
If $w,v\in S_N$ and their Fourier coefficients $\hat w_k$ and $\hat v_k$, $k=-2N,\cdots,2N$, are stored in the computer (with $\hat w_k=\hat v_k=0$ for $N<|k|\le 2N$), then the values $w(x_n)$ and $v(x_n)$, $n=-2N,\dots,2N$, can be computed exactly by using \eqref{def:iiN2} and FFT. 
Since $(wv)_k=0$ for $|k|>2N$, it follows that $wv=I_{2N}(wv)$. If we denote by $\mathcal{F}_k[v]$ the $k$th Fourier coefficient of the function $v$, then 
$$ 
\mathcal{F}_k[wv] 
= 
\sum\limits_{n=-2N}^{2N} e^{- i k x_n }  w(x_n) v(x_n) ,\quad
k=-2N,\dots,2N, 
$$ 
which can also be computed exactly by using FFT. 
Therefore, if we denote by $\Pi_N: L^2(\T)\rightarrow L^2(\T)$ the projection operator defined by 
$$
\mathcal{F}_k[\Pi_N f]=
\left\{
\begin{aligned}
&\hat f_k &&\mbox{for}\,\,\, |k|\le N, \\
&0 &&\mbox{for}\,\,\, |k|> N ,
\end{aligned}
\right.
$$
then the cost of computing the Fourier coefficients of $\Pi_N(wv)\in S_N$ from the Fourier coefficients of $w,v\in S_N$ is $O(N\ln N)$. 


For any positive integer $L$, let $t_n=n\tau$, $n=0,1,\dots,L$, be a partition of the time interval $[0,T]$ with stepsize $\tau=T/L$. The fully discrete low-regularity integrator for the NLS equation \eqref{model} to be constructed in this paper is: For given $u_{\tau,N}^n\in S_N$ compute $u^{n+1}_{\tau,N}\in S_N$ by 
\begin{align}\label{numerical}
\begin{aligned}
&u^{n+1}_{\tau,N}=\Psi(u^n_{\tau,N}) \quad\mbox{for}\,\,\, n=0,1\ldots,L-1 , \\
&\mbox{with $u^0_{\tau,N} = \Pi_N I_{2N}u^0 \in S_N$} ,  
\end{aligned}
\end{align}
where 
\begin{align}\label{psi}
\Psi(f) 
:=\, 
& \fe^{i\tau(-2\lambda P_N\partial_x^{-1}-2\lambda M_N+ \partial_x^2)}f
+ (1-\fe^{-2i\lambda\tau M_N}) \Pi_0 f 
-i\lambda \tau \Pi_0\big[ \Pi_N(|f|^2)f\big]
\notag\\
 & + \lambda \partial_x^{-1}\Pi_N
          \big[ (\fe^{i\tau\partial_x^2} f)
              \cdot \partial_x^{-1}\Pi_N
                (|\fe^{i\tau\partial_x^2} f|^2 ) \big]
                -\lambda \fe^{i\tau\partial_x^2} \partial_x^{-1}\Pi_N
          \big[ f  \cdot \partial_x^{-1} \Pi_N (|f |^2 ) \big]\notag\\     
 & -\frac{\lambda}{2} \Big[ \partial_x^{-2}\Pi_N
             \Big((\fe^{-i\tau\partial_x^2}\bar f\, )
              \, \fe^{i\tau\partial_x^2}\Pi_N (f^2) \Big)
              - \fe^{i \tau\partial_x^2} \partial_x^{-2}\Pi_N
          \Big( \bar f\, \Pi_N(f^2) \Big) \Big] \notag \\ 
 & -\frac{\lambda}{2}
\fe^{i \tau\partial_x^2} \partial_x^{-1}\Pi_N
          \Big[\partial_x \bar f \Big(
              \,  \fe^{-i\tau\partial_x^2} \Pi_N\big[
                (\fe^{i \tau \partial_x^2}\partial_x^{-1} f )^2\big]    - \Pi_N \big[ (\partial_x^{-1} f)^2\big] \Big)  \Big] \notag\\
&
-i\lambda\tau \fe^{i \tau \partial_x^2}\partial_x^{-1} \Pi_N
\big( \partial_x\bar f \, \Pi_N(f^2)\big) \notag \\ 
& +2i\lambda \tau \Pi_0 f\, \fe^{i \tau\partial_x^2} \partial_x^{-1} \Pi_N \left(\partial_x\bar f \, f\right)
-i\lambda\tau (\Pi_0 f)^2 \fe^{i \tau\partial_x^2} \Pi_{\ne 0} \bar f 
\qquad\mbox{for}\,\,\, f\in S_N ,  
\end{align}
and
\begin{align}\label{M_N+P_N}
M_N & 
= \Pi_0(|u^0_{\tau,N}|^2)
\quad\mbox{and}\quad 
P_N 
=\Pi_0\big(u^0_{\tau,N}\partial_x \overline{u^0_{\tau,N}}\big) 
\end{align}
are the approximate mass and momentum, respectively. 
By using \eqref{def:iiN} with FFT, the initial value $u^0_{\tau,N} =\Pi_N I_{2N}u^0$ can be obtained with cost $O(N\ln N)$. Then, at every time level, the method only requires computing several functions in the following forms:
\begin{align*}
&\mbox{\footnotesize$\bullet$}\quad \fe^{i\tau\left(-2M_N-2P_N\partial_x^{-1}+\partial_x^2\right)} f
,\,\,\,
\fe^{\pm i\tau\partial_x^2} f
\,\,\,\mbox{and}\,\,\,
\partial_x^{-1} f 
\,\,\,\mbox{for some given function}\,\,\, f \in S_N ,  \\
&\mbox{\footnotesize$\bullet$}\,\,\,\Pi_N(fg) \,\,\,
\mbox{for some given functions}\,\,\, f, g \in S_N , 
\end{align*}
where  
\begin{align*}
& 
\mathcal{F}_k [ \fe^{i\tau\left(-2M_N-2P_N\partial_x^{-1}+\partial_x^2\right)} f]
= \left\{
\begin{aligned}
& \fe^{-2M_N i\tau } \hat f_0 &&\mbox{for}\,\,\, k=0,\\ 
&\fe^{i\tau\left(-2M_N-2P_N(ik)^{-1}- k^2 \right)} \hat f_k&&\mbox{for}\,\,\, k\neq 0.
\end{aligned}
\right. \\
&\mathcal{F}_k [\fe^{\pm i\tau\partial_x^2}f]
=  \fe^{\mp i\tau k^2} \hat f_k 
\quad\mbox{and}\quad
\partial_x^{-1} f
=  \left\{
\begin{aligned}
& 0  &&\mbox{for}\,\,\, k=0,\\ 
&  (ik)^{-1} \hat f_k &&\mbox{for}\,\,\, k\neq 0.
\end{aligned}
\right.
\end{align*}
Hence, the computational cost is $O(N\ln N)$ at every time level. 


%
%

%

The main theoretical result of this paper is the following theorem.

\begin{theorem}\label{main:thm1}
If $u^0\in H^{1}(\T)$ then there exist positive constants $\tau_0$, $N_0$ and $C$ such that for $\tau\leq\tau_0$ and $N\ge N_0$ the numerical solution given by \eqref{numerical}--\eqref{psi} has the following error bound:
\begin{equation}\label{error-N-1}
\max_{1\le n\le L}  \|u(t_n,\cdot)-u^{n}_{\tau,N}\|_{L^2}
  \le C \big(\tau\sqrt{\ln (1/\tau) } +N^{-1}\big) ,
\end{equation} 
where the constants $\tau_0$, $N_0$ and $C$ depend only on $T$ and $\|u^0\|_{H^1}$. 
\end{theorem}
  
The rest of this paper is devoted to the construction of the method \eqref{numerical}--\eqref{psi} and the proof of Theorem \ref{main:thm1}. 

\begin{remark}\label{Remark}
{\upshape
The analysis in this article can be easily extended to proving higher-order convergence of the spatial discretisation method when the initial data is smoother. Namely, for $u^0\in H^{s}(\T)$ with $s>1$, the error bound of of the proposed method should become
\begin{equation}\label{error-N-s}
\max_{1\le n\le L}  \|u(t_n,\cdot)-u^{n}_{\tau,N}\|_{L^2}
  \le C \big(\tau +N^{-s}\big) . 
\end{equation} 
The proof of this result (with smoother initial data) is easier than the proof of Theorem \ref{main:thm1} and therefore omitted. The convergence results in \eqref{error-N-1} and \eqref{error-N-s} are illustrated by the numerical experiments in section \ref{sec:numerical} for $s=1$ and $s=2$, respectively.
}
\end{remark}

\section{Notation and technical tools}
\label{section:tool}

In this section we introduce the basic notation and technical lemmas to be used in analysing the error of the numerical method to be constructed. 

\subsection{Notation}\label{subsec1}
The inner product and norm on $L^2(\T)$ are denoted by  
$$
( f,g ) = 
\int_\T f(x) \overline{g(x)}\,dx
\quad\mbox{and}\quad
\|f\|_{L^2}= \sqrt{( f,f )} ,\,\,\,\mbox{respectively}. 
$$
The norm on the Sobolev space $H^s(\T)$, $s\in\R$, is denoted by 
$$
\big\|f\big\|_{H^s}^2
=2\pi\sum_{k\in \Z}(1+ k ^2)^s |\hat{f}_k|^2 . 
$$
For a function $f:[0,T]\times\T\rightarrow \C$ we denote by $\|f\|_{L^p(0,T;H^s)}$ its space-time Sobolev norm, defined by 
$$
\|f\|_{L^p(0,T;H^s)}
= 
\left\{\begin{aligned}
&\bigg(\int_0^T \|f(t)\|_{H^s}^p d t\bigg)^{\frac{1}{p}} && \mbox{for}\,\,\, p\in[1,\infty) ,\\[5pt] 
&{\rm ess\!\!}\sup \limits_{t\in[0,T]\,\,\,}\!\! \|f(t)\|_{H^s} && \mbox{for}\,\,\, p=\infty. 
\end{aligned}\right. 
$$

The Fourier coefficients of a function $f$ on $\T$ are denoted by $\mathcal{F}_k[f]$ or simply $\hat{f}_k$, defined by 
$$
\hat{f}_k = \frac{1}{2\pi}\int_{\T} e^{- i   kx }f( x )\,d x \quad\mbox{for}\,\,\, k\in\bZ .
$$
The Fourier inversion formula is given by 
$$
f( x )=\sum_{k\in \Z} e^{ i  kx } \hat{f}_k.
$$
The Fourier coefficients are known to have the following properties:
\begin{align*}
\begin{aligned}
\|f\|_{L^2}^2
 & =2\pi \sum\limits_{k\in \Z}\big|\hat{f_k}\big|^2 && \mbox{(Plancherel identity)}; \\
\mathcal{F}_k[fg]  &=\sum\limits_{k_1\in\Z}
  \hat{f}_{k - k _1}\hat{g}_{k _1}  && \mbox{(Convolution)}.
  \end{aligned}
\end{align*}

For any function $\sigma:\Z\rightarrow\C$ such that $|\sigma(k)|\le C_\sigma(1+|k|)^m$ for some constants $C_\sigma$ and $m\ge 0$, we denote by $\sigma(i^{-1}\partial_x):H^s(\T)\rightarrow H^{s-m}(\T)$ the operator defined by 
$$
\sigma(i^{-1}\partial_x)f = \sum_{k\in \Z} \sigma(k) \hat f_k e^{ikx} . 
$$
For abbreviation, we denote 
$$\l k \r= (1+k^2)^{\frac{1}{2}} \quad\mbox{and}\quad J^s=  \l i^{-1} \partial_x \r^s ,$$
which imply that 
$$
\big\|f\big\|_{H^s}^2=\big\|J^sf\big\|_{L^2}^2 \quad\mbox{and}\quad
\widehat{J^s f } _k = \l k \r^s \hat f_k . 
$$
Moreover, we denote by $\partial_x^{-1}:H^s(\T)\rightarrow H^{s+1}(\T)$, $s\in\R$, the operator such that 
\begin{equation}\label{def:px-1}
\mathcal{F}_k[\partial_x^{-1}f]
=\Bigg\{ \aligned
    &(i k )^{-1}\hat{f}_k,\quad &\mbox{when }  k \ne 0,\\
    &0,\quad &\mbox{when }  k = 0.
   \endaligned
\end{equation}

We denote by $A\lesssim B$ or $B\gtrsim A$ the statement 
$A\leq CB$ for some constant $C>0$.  
The value of $C$ may depend on $T$ and $\|u^0\|_{H^1}$, and may be different at different occurrences, but is always independent of $\tau$, $N$ and $n$. 
The notation $A\sim B$ means that $A\lesssim B\lesssim A$. 

We denote by $O(Y)$ any quantity $X$ such that $X\lesssim Y$. 
For any function $\sigma:\Z^{m+1}\rightarrow\C$ and $w\in H^1(\T)$ we denote by $\mathcal T_m(\sigma; w)$ the class of functions $f\in L^2(\T)$ such that 
\begin{align}\label{def:TmM}
\hat{f}_k \lesssim 
\sum\limits_{ k _1+\cdots+ k _m = k}
| \sigma( k, k _1,\cdots, k _m)|\>|\hat{w}_{ k _1}| \cdots |\hat w_{ k _m}|  
\quad\forall\, f\in \mathcal T_m(\sigma; w). 
\end{align}
If $F=\int_{t_1}^{t_2} f(t)dt$ for some function $f(t)\in \mathcal T_m(\sigma; v(t))$, then we simply denote 
\begin{align}\label{def:TmM2}
F\in \int_{t_1}^{t_2}   \mathcal T_m(\sigma; v(t)) dt . 
\end{align}

\subsection{Two technical lemmas}\label{subsec3}
We will use the following version of the Kato--Ponce inequalities, which was originally proved 
 in \cite{Kato-Ponce}
and subsequently improved to cover the endpoint case in 
\cite{BoLi-KatoPonce, Li-KatoPonce}.

\begin{lemma}[The Kato--Ponce inequalities] \label{lem:kato-Ponce} 
$\,$
\begin{itemize}
  \item[(i)]
If $ s>\frac 12$ and $f,g\in H^{s}(\T)$ then
\begin{align*}
\|fg\|_{H^s}\lesssim \|f\|_{H^s}\|g\|_{H^{s}}.
\end{align*}
  \item[(ii)]
If $s\ge 0, s_1>\frac 12$, $f\in H^{s+s_1}(\T)$ and $g\in H^{s}(\T)$, then
\begin{align*}
\|fg\|_{H^s} \lesssim \|f\|_{H^{s+s_1}}\|g\|_{H^{s}}.
\end{align*}
\end{itemize}
\end{lemma}

In addition to Lemma \ref{lem:kato-Ponce} we also need the following results, which are consequences of the Kato--Ponce inequalities.
\begin{lemma}\label{lem:kato-Ponce-1} 
$\,$
\begin{itemize}
  \item[(i)]
If $s>\frac12$ and $f,g\in H^s(\T)$ then 
\begin{align*}
\|J^{-1}(Jf \, g)\|_{H^s}\lesssim  \|f\|_{H^s}\|g\|_{H^s}.
\end{align*}
\item[(ii)]
If $f,g \in H^1(\T)$ then 
\begin{align*}
\|J^{-1}(Jf \, g)\|_{L^2}\lesssim \min\big\{\|f\|_{L^2}\|g\|_{H^1},\|g\|_{L^2}\|f\|_{H^1}\big\}.
\end{align*}
\end{itemize}
\end{lemma}
\begin{proof}

(i) The desired inequality is equivalent to $\|J^{s-1}(Jf\, g)\|_{L^2}\lesssim  \|f\|_{H^s}\|g\|_{H^s}$. 
By the duality between $L^2(\T)$ and itself, it suffices to prove  
$$
(J^{s-1}(Jf\, g),h) \lesssim  \|f\|_{H^s}\|g\|_{H^s}\|h\|_{L^2} \quad\forall\, h\in L^2(\T) ,  
$$
which is equivalent to 
\begin{align*}
\sum\limits_{k}\sum\limits_{k_1+k_2=k} \l k\r^{s-1} \l k_1\r \hat f_{k_1} \hat g_{k_2} \overline{\hat h_k}
\lesssim  \|f\|_{H^s}\|g\|_{H^s}\|h\|_{L^2}.
\end{align*}
Since the term corresponding to $k=0$ satisfies 
\begin{align*}
\sum\limits_{k_1+k_2=0}  \l k_1\r \hat f_{k_1}\hat g_{k_2} \overline{\hat h_0}
&= \sum\limits_{k_1}  \l k_1\r^{\frac12} \hat f_{k_1} \l -k_1\r^{\frac12} \hat g_{-k_1} \overline{\hat h_0}\\
&\lesssim  \| (\l k_1\r^{\frac12} \hat f_{k_1})_{k_1\in\Z} \|_{l^2} 
\| (\l -k_1\r^{\frac12} \hat g_{-k_1})_{k_1\in\Z} \|_{l^2}  |\hat h_0| \\[5pt] 
&\lesssim  \|f\|_{H^{\frac12}}\|g\|_{H^{\frac12}}\|h\|_{L^1} \\
&\lesssim  \|f\|_{H^s}\|g\|_{H^s}\|h\|_{L^2} \quad\mbox{when}\,\,\, s>\frac12 , 
\end{align*}
we only need to prove the following result: 
\begin{align*}
\sum\limits_{k\ne 0}\sum\limits_{k_1+k_2=k}|k|^{s-1}|k_1| \hat f_{k_1}\hat g_{k_2} 
\overline{\hat h_k}
\lesssim  \|f\|_{H^s}\|g\|_{H^s}\|h\|_{L^2}.
\end{align*}
To this end, we decompose the left-hand side of the inequality above into two parts, i.e.,  
\begin{align}
\begin{aligned}
&\sum\limits_{k\ne 0}\sum\limits_{k_1+k_2 = k}|k|^{s-1}|k_1| \hat f_{k_1}\hat g_{k_2} \overline{\hat h_k}  \\ 
&\lesssim  \sum\limits_{k\ne 0}\sum\limits_{{\substack{k_1+k_2 = k \\   |k_1|\le 10|k|}}}|k|^{s-1}|k_1| |\hat f_{k_1}| |\hat g_{k_2}| |\hat h_k| 
+\sum\limits_{k\ne 0}\sum\limits_{{\substack{k_1+k_2 = k \\   |k_1|> 10|k|}}}|k|^{s-1}|k_1| |\hat f_{k_1}| |\hat g_{k_2}| |\hat h_k| .
\end{aligned}
\label{10.03-2}
\end{align}

The first term on the right-hand side of \eqref{10.03-2} can be estimated by using Plancherel's identity and Lemma \ref{lem:kato-Ponce} as follows: 
\begin{align*}
 \sum\limits_{k\ne 0}\sum\limits_{{\substack{k_1+k_2 = k \\   |k_1|\le 10|k|}}}|k|^{-1+s}|k_1| |\hat f_{k_1}| |\hat g_{k_2}| |\hat h_k|
 \lesssim &
  \sum\limits_{k\ne 0}\sum\limits_{{\substack{k_1+k_2 = k \\   |k_1|\le 10|k|}}}|k|^{s} |\hat f_{k_1}||\hat g_{k_2}|| \hat h_k| \\
\lesssim & (J^s(\tilde f\tilde g),\tilde h) \\
 \lesssim &
 \big\|\tilde f \tilde g\big\|_{H^s}\|\tilde h\|_{L^2}
 \lesssim 
 \|\tilde f\|_{H^s}\| \tilde g\|_{H^s}\|\tilde h\|_{L^2} , 
\end{align*}
where $\tilde f$, $\tilde g$ and $\tilde h$ are functions with Fourier coefficients $|\hat f_k|$, $|\hat g_k|$ and $|\hat h_k|$, respectively. Since 
$$
\|\tilde f\|_{H^s} \sim \|f\|_{H^s} ,
\quad
\|\tilde g\|_{H^s} \sim \|g\|_{H^s}
\quad\mbox{and}\quad
\|\tilde h\|_{L^2} \sim \|h\|_{L^2} ,  
$$
it follows that
\begin{align*}
 \sum\limits_{k\ne 0}\sum\limits_{{\substack{k_1+k_2 = k \\   |k_1|\le 10|k|}}}|k|^{-1+s}|k_1| |\hat f_{k_1}| |\hat g_{k_2}| |\hat h_k|
 \lesssim 
 \| f\|_{H^s}\|  g\|_{H^s}\| h\|_{L^2} . 
\end{align*}

In the second term on the right-hand side of \eqref{10.03-2}, we have $|k_1|\sim |k_2|> |k|$. For $s>\frac12$ we have
$$|k|^{s-1} |k_1| =|k|^{-s} |k|^{2s-1} |k_1|  \le |k|^{-s} |k_1|^{2s} \sim |k|^{-s} |k_1|^{s} |k_2|^{s}$$ and therefore  
\begin{align*}
\sum\limits_{k\ne 0}\sum\limits_{{\substack{k_1+k_2 = k\\   |k_1|> 10|k|}}}|k|^{s-1}|k_1| |\hat f_{k_1}| |\hat g_{k_2}| |\hat h_k|
\lesssim &
 \sum\limits_{k\ne 0}\sum\limits_{{\substack{k_1+k_2 = k \\   |k_1|> 10|k|}}}|k|^{-s} |k_1|^s |k_2|^s |\hat f_{k_1}||\hat g_{k_2} ||\hat h_k|\\
\lesssim &
\sum\limits_{k\ne 0} \mathcal{F}_k[J^s\tilde f\, J^s\tilde g] |k|^{-s}|\hat h_k |\\
\lesssim &
\max_{k}| \mathcal{F}_k[J^s\tilde f\, J^s\tilde g]| \sum\limits_{k\ne 0} |k|^{-s}|\hat h_k| \\ 
\lesssim &
\|J^s\tilde f\, J^s\tilde g\|_{L^1} \| (|k|^{-s})_{0\ne k\in\Z} \|_{l^2} \| (|\hat h_k|)_{0\ne k\in\Z} \|_{l^2}  \\ 
\lesssim &
\|J^s\tilde f\|_{L^2} \|J^s\tilde g\|_{L^2} \|\tilde h\|_{L^2} \\ 
 \lesssim &
 \|f\|_{H^s}\|g\|_{H^s} \|h\|_{L^2}.
\end{align*}
This completes the proof of (i).

(ii) Similarly as (i),  it suffices to prove  
\begin{align*}
\sum\limits_{k\ne 0}\sum\limits_{k_1+k_2 = k}|k|^{-1}|k_1| \hat f_{k_1}\hat g_{k_2} \overline{\hat h_k}
\lesssim \min( \|f\|_{L^2}\|g\|_{H^1}, \|f\|_{H^1}\|g\|_{L^2} ) \|h\|_{L^2} 
\quad\forall\, h\in L^2(\T). 
\end{align*}
In view of the proof of (i), we can assume $\hat f_k\ge 0$, $\hat g_k\ge 0$ and $\hat h_k\ge 0$ without loss of generality  
(otherwise we can replace $f$, $g$ and $h$ by $\tilde f$, $\tilde g$ and $\tilde h$, respectively, in the estimates below). Then
\begin{align}\label{Lemma2.2-ii-ie}
\begin{aligned}
&\sum\limits_{k\ne 0}\sum\limits_{k_1+k_2 = k}|k|^{-1}|k_1| \hat f_{k_1}\hat g_{k_2} \hat h_k  \\
&\lesssim   \sum\limits_{k\ne 0}\sum\limits_{{\substack{k_1+k_2 = k\\   |k_1|\le 10|k|}}}|k|^{-1}|k_1| \hat f_{k_1}\hat g_{k_2} \hat h_k 
+\sum\limits_{k\ne 0}\sum\limits_{{\substack{k_1+k_2 = k\\   |k_1|> 10|k|}}}|k|^{-1}|k_1| \hat f_{k_1}\hat g_{k_2} \hat h_k.
\end{aligned}
\end{align}
The first term on the right-hand side of \eqref{Lemma2.2-ii-ie} can be estimated by using Plancherel's identity and Lemma \ref{lem:kato-Ponce}: 
\begin{align*}
 \sum\limits_{k\ne 0}\sum\limits_{{\substack{k_1+k_2 = k\\   |k_1|\le 10|k|}}}|k|^{-1}|k_1| \hat f_{k_1}\hat g_{k_2} \hat h_k  \lesssim &
  \sum\limits_{k\ne 0}\sum\limits_{{\substack{k_1+k_2 = k\\   |k_1|\le 10|k|}}}\hat f_{k_1}\hat g_{k_2} \hat h_k\\
  \lesssim &
  \sum\limits_{k\ne 0} \mathcal{F}_k[fg]  \hat h_k\\ 
\lesssim &
  \| (\mathcal{F}_k[fg])_{0\ne k\in\Z} \|_{l^2} \| (\hat h_k)_{0\ne k\in\Z} \|_{l^2}  \\ 
 \lesssim &
 \big\|fg\big\|_{L^2}\|h\|_{L^2} \\
  \lesssim &
 \min( \|f\|_{L^2}\|g\|_{L^\infty},\|f\|_{L^\infty}\|g\|_{L^2} ) \|h\|_{L^2}\\ 
  \lesssim &
 \min(\|f\|_{L^2}\|g\|_{H^{1}},\|f\|_{H^{1}}\|g\|_{L^2})\|h\|_{L^2} .
\end{align*}
In the second term on the right-hand side of \eqref{Lemma2.2-ii-ie} we have $|k_1|\sim |k_2|>k$. On the one hand, we have
\begin{align*}
\sum\limits_{k\ne 0}\sum\limits_{{\substack{k_1+k_2=k \\   |k_1|> 10|k|}}}|k|^{-1}|k_1| \hat f_{k_1}\hat g_{k_2} \hat h_k
\lesssim &
\sum\limits_{k\ne 0}\sum\limits_{k_2} |k|^{-1} |k_2| \hat f_{k-k_2}\hat g_{k_2}  \hat h_k \\
\lesssim &
\bigg( \sup_{k} \sum_{k_2}|\hat f_{k-k_2}|^2 \bigg)^{\frac12} \bigg(\sum_{k_2} |k_2|^2  |\hat g_{k_2}|^2 \bigg)^{\frac12} \sum\limits_{k\ne 0}|k|^{-1}  \hat h_k \\
\lesssim & \|f\|_{L^2}\|g\|_{H^1}\|h\|_{L^2}.
\end{align*}
On the other hand, we have 
\begin{align*}
\sum\limits_{k\ne 0}\sum\limits_{{\substack{k=k_1+k_2\\   |k_1|> 10|k|}}}|k|^{-1}|k_1| \hat f_{k_1}\hat g_{k_2} \hat h_k
\lesssim &
\sum\limits_{k\ne 0}\sum\limits_{k_1} |k|^{-1} |k_1| \hat f_{k_1}\hat g_{k-k_1}  \hat h_k \\
\lesssim &
\bigg(\sum_{k_1}|k_1|^2|\hat f_{k_1}|^2 \bigg)^{\frac12} \bigg( \sup_{k}\sum_{k_1}   |\hat g_{k-k_1}|^2 \bigg)^{\frac12} \sum\limits_{k\ne 0}|k|^{-1}  \hat h_k \\
\lesssim & \|f\|_{H^1}\|g\|_{L^2}\|h\|_{L^2}.
\end{align*}
This completes the proof of (ii).
\end{proof}


\section{Construction of the method through analysing consistency error}
\label{section:construction}

In this section we construct the numerical method based on twisted variables and Duhamel's formula through analysing the consistency errors in approximating the exponential integrals using harmonic analysis techniques. 
For readers' convenience, we present the derivation of the numerical method in subsection \ref{section:derivation} and defer the technical estimates to subsection \ref{section:consistency}. 

\subsection{Construction of the numerical method}\label{section:derivation}

\noindent 
As mentioned in the introduction section and the beginning of section \ref{section:result}, the NLS equation \eqref{model} has a unique solution $u\in C([0,T];H^1(\T))$ satisfying the Duhamel's formula: 
\begin{align}\label{Duhamel-u}
u(t_{n+1})=\fe^{i\tau\partial_x^2}u(t_n)
 -  i \lambda \int_0^\tau \fe^{i\left(t_{n+1}-(t_n+s)\right)\partial_x^2}
   |u(t_n+s)|^2u(t_n+s) \,ds ,
\end{align} 
as well as the mass and momentum conservations \eqref{mass}--\eqref{momentum}. The norm $\|u\|_{C([0,T];H^1(\T))} $ is bounded by a constant depending on $\|u^0\|_{H^1}$; see \cite{Bo}. 

Let $v(t):=\fe^{-it\partial_x^2}u(t)$ be the twisted variable. Then $v\in C([0,T];H^1(\T))$ satisfies $\|v\|_{C([0,T];H^1(\T))}=\|u\|_{C([0,T];H^1(\T))}$ and the following conservation laws simiarly as $u$, i.e., 
\begin{enumerate}
\item
Mass conservation: 
\begin{align}
\frac1{2\pi}\int_\T |v(t,x)|^2\,d x = \frac1{2\pi}\int_\T |u(t,x)|^2\,d x = M \quad\mbox{for}\,\,\, t>0 .  \label{mass-v}
\end{align}

\item
Momentum conservation: 
\begin{align}
\frac1{2\pi}\int_\T v(t,x)\partial_x\bar v(t,x)\,d x = \frac1{2\pi}\int_\T u(t,x)\partial_x\bar u(t,x)\,d x 
=P \quad\mbox{for}\,\,\, t>0 .  \label{momentum-v}
\end{align}
\end{enumerate}
Applying the operator $\fe^{-it_{n+1} \partial_x^2}$ to the identity \eqref{Duhamel-u}, we obtain 
\begin{align}\label{solution}
v(t_{n+1})=v(t_n)-i \lambda \int_0^\tau \fe^{-i(t_n+s)\partial_x^2}
    \big[|\fe^{i(t_n+s)\partial_x^2}v(t_n+s)|^2\,\fe^{i(t_n+s)\partial_x^2}v(t_n+s)\big]\,ds.
\end{align}
The Fourier coefficients of both sides of \eqref{solution} should be equal, i.e.,  
\begin{align}\label{solution-F}
\hat v_{ k }(t_{n+1})
  =\hat v_{ k }(t_n)
     -i \lambda \int_0^\tau\sum\limits_{ k _1+ k _2+ k _3 = k }\fe^{i(t_n+s)\phi}
        \>\hat{ \bar v}_{ k _1}(t_n+s)\hat v_{ k _2}(t_n+s)\hat v_{ k _3}(t_n+s)\,ds , 
\end{align}
with a phase function  
\begin{align*}
\phi 
=\phi( k, k _1, k _2, k _3) 
= k ^2+ k _1^2- k _2^2- k _3^2.
\end{align*}

Replacing $\tau$ and $s$ in \eqref{solution-F} by $s$ and $\sigma$, respectively, we have 
\begin{align}\label{solution-F2}
\hat v_{ k }(t_{n}+s)
  =\hat v_{ k }(t_n)
     -i \lambda \int_0^{s}\sum\limits_{k _1+ k _2+ k _3 = k }\fe^{i(t_n+\sigma)\phi}
        \>\hat{ \bar v}_{ k _1}(t_n+s)\hat v_{ k _2}(t_n+\sigma)\hat v_{ k _3}(t_n+\sigma)\,d\sigma . 
\end{align}
In view of \eqref{solution-F2} and the definition of $\mathcal T_m(M; v)$ in \eqref{def:TmM}, we have 
\begin{align}\label{vtn+s-vtn}
v(t_n+s)-v(t_n) \in \int_0^s \, \mathcal T_3(1;v(t_n+\sigma)) d\sigma .
\end{align} 
As a result, \eqref{solution-F} can be written as 
\begin{align}\label{v-1st-phi}
\hat v_{ k }(t_{n+1}) =\hat v_{ k }(t_n)
    -i \lambda \sum\limits_{ k _1+ k _2+ k _3 = k } 
       \>\hat{ \bar v}_{ k _1}(t_n)\hat v_{ k _2}(t_n)\hat v_{ k _3}(t_n) \int_0^\tau \fe^{i(t_n+s)\phi}\,ds+\>\hat{\mathcal R}_{1,k} ,
\end{align}
with
\begin{align*}
&\hat{\mathcal R}_{1,k} \\ 
&= -i \lambda \int_0^\tau\sum\limits_{ k _1+ k _2+ k _3 = k }\fe^{i(t_n+s)\phi}
        \>\big(\hat{ \bar v}_{ k _1}(t_n+s)\hat v_{ k _2}(t_n+s)\hat v_{ k _3}(t_n+s)
        -\hat{ \bar v}_{k _1}(t_n)\hat v_{ k _2}(t_n)\hat v_{ k _3}(t_n)\big)\,ds \\
&\in \int_0^\tau\int_0^s \, \mathcal T_5(1;v(t_n+\sigma)) d\sigma ds , 
\end{align*}
where the last inclusion is based on the definition in \eqref{def:TmM2}. 
If $\mathcal{R}_1$ denotes the function with Fourier coefficients 
$\hat{\mathcal R}_{1,k}$, then the relation above implies that (according to Lemma \ref{lem:Rj} (i) of the next subsection) 
\begin{align}\label{est:R1}
\big\|\mathcal R_1\big\|_{H^1}\lesssim \tau^2\|v\|_{L^\infty_tH^1_x}^5.
\end{align}
This term will be dropped in our numerical scheme. 

In the following, we approximate the second term on the right-hand side of \eqref{v-1st-phi} by expressions that can be evaluated efficiently with FFT. To this end, we consider the three cases $k=0$, $|k|>N$ and $0\ne |k|\le N$, separately. 

{\sc Case} 1: $k=0$.  
In this case, \eqref{v-1st-phi} reduces to 
\begin{align}\label{zeromode-1}
\hat v_0(t_{n+1})
=&\hat v_0(t_n) -i\lambda\sum\limits_{k_1+k_2+k_3=0} 
\hat{ \bar v}_{k_1}(t_n)\hat v_{k_2}(t_n)\hat v_{k_3}(t_n) 
\int_0^\tau \fe^{i(t_n+s)(k_1^2-k_2^2-k_3^2)}\,ds 
+ \hat{\mathcal R}_{1,0} \\
=&
       \hat v_0(t_n)-i\lambda \tau\sum\limits_{k_1+k_2+k_3=0} \fe^{it_n(k_1^2-k_2^2-k_3^2)}
       \>\hat{ \bar v}_{k_1}(t_n)\hat v_{k_2}(t_n)\hat v_{k_3}(t_n)
+ \hat{\mathcal R}_{1,0}+\hat{\mathcal R}_{2,0}\notag\\ 
=&
\hat v_0(t_n)-i\lambda \tau \Pi_0\big(\big|\fe^{it_n \partial_{x}^2}v(t_n)\big|^2\fe^{it_n \partial_{x}^2}v(t_n)\big) 
+\hat{\mathcal R}_{1,0}+\hat{\mathcal R}_{2,0}\notag\\[5pt]
=&
\hat v_0(t_n)-i\lambda \tau \Pi_0\big[\Pi_N\big(\big|\fe^{it_n \partial_{x}^2}v(t_n)\big|^2)\fe^{it_n \partial_{x}^2}v(t_n)\big]
+\hat{\mathcal R}_{1,0}+\hat{\mathcal R}_{2,0} +  \hat{\mathcal R}_{2,0}^* , 
\end{align} 
with 
\begin{align*}
\hat{\mathcal R}_{2,k}
&= -i\lambda\sum\limits_{k_1+k_2+k_3=k} 
\widehat{\bar v}_{k_1}(t_n)\hat v_{k_2}(t_n)\hat v_{k_3}(t_n) 
\int_0^\tau (\fe^{i(t_n+s)(k_1^2-k_2^2-k_3^2)}-\fe^{it_n(k_1^2-k_2^2-k_3^2)})\,ds  , \\[3pt] 
{\mathcal R}_{2}^*
&= -i\lambda \tau \big[(1-\Pi_N)\big(\big|\fe^{it_n \partial_{x}^2}v(t_n)\big|^2 \big)\big]\fe^{it_n \partial_{x}^2}v(t_n) \in \tau \fe^{it_n \partial_{x}^2}v(t_n) \mathcal{T}_2(1_{>N};v(t_n)) ,  
\end{align*}
where 
\begin{align}\label{def-1N}
(1_{>N})_k 
= 1_{|k|>N} 
= 
\left\{
\begin{aligned}
&0 &&\mbox{for}\,\,\, |k|\le N,\\
&1 &&\mbox{for}\,\,\, |k|> N .  
\end{aligned}
\right.
\end{align}
Since $k_1+k_2+k_3=0$, it follows that there holds 
$
k_1^2-k_2^2-k_3^2=2k_2 k_3 
$ 
and therefore 
\begin{align*} 
\int_0^\tau \left(\fe^{i(t_n+s)(k_1^2-k_2^2-k_3^2)}-\fe^{it_n(k_1^2-k_2^2-k_3^2)}\right)\,ds = \tau^2 O(k_2k_3).
\end{align*}
As a result, the function $\mathcal R_2$ (with Fourier coefficients $\hat{\mathcal R}_{2,k}$) satisfies that $\mathcal R_2\in \tau^2 \mathcal T_3(k_2k_3;v(t_n))$ in view of the definition in  \eqref{def:TmM}. 
According to Lemma \ref{lem:Rj} (i)--(ii) of the next subsection, $\mathcal R_2$ and $\mathcal R_2^*$ satisfy the following estimates: 
\begin{align}\label{est:R2}
|\hat{\mathcal R}_{2,0}|
&\lesssim \tau^2\|v\|_{L^\infty_tH^1_x}^3 ,\\[3pt]
|\hat{\mathcal R}_{2,0}^*|
&\lesssim
\|\mathcal R_{2}^*\|_{L^1}
\lesssim 
\tau 
 \|\fe^{it_n \partial_{x}^2}v(t_n)\|_{L^\infty} \|(1-\Pi_N)(|\fe^{it_n \partial_{x}^2}v(t_n)|^2 )\|_{L^2} 
\lesssim
\tau N^{-1} \|v\|_{L^\infty_tH^1_x}^3  .
\end{align} 
The two terms $\hat{\mathcal R}_{2,0}$ and $\hat{\mathcal R}_{2,0}^*$ will be dropped in our numerical scheme. 



{\sc Case} 2: $|k|> N$. Let $\mathcal R_3$ be the function with Fourier coefficients    
$$ 
\hat{\mathcal R}_{3,k}=-1_{|k|>N} \, i\lambda \sum\limits_{k_1+k_2+k_3=k} 
\hat{ \bar v}_{ k _1}(t_n) \hat v_{ k _2}(t_n) \hat v_{ k _3}(t_n) 
\int_0^\tau \fe^{i(t_n+s)\phi}\,ds . 
$$ 
Then 
\begin{align*}
\mathcal R_3 \in\tau \,  \mathcal T_3\left(1_{>N};v(t_n)\right) .  
\end{align*}
Lemma \ref{lem:Rj} (i) of the next subsection implies that 
\begin{align}\label{est:R3}
\big\|\mathcal R_3\big\|_{H^s} \lesssim \tau N^{-1+s}\|v\|_{L^\infty_tH^1_x}^3
\quad \mbox{for}\,\,\, s\in[0,1] . 
\end{align}
This term will be dropped in the numerical scheme. 

{\sc Case} 3: $0\ne |k|\le N$. 
By using the  identity 
$$1=\frac{(k_1+k_2)+(k_1+k_3)-k_1}{k} $$
and symmetry between $k_2$ and $k_3$, 
we can decompose the second term on the right-hand side of \eqref{v-1st-phi} into two parts, i.e., 
\begin{subequations}\label{v-2-0}
\begin{align}
\hat v_{ k }(t_{n+1}) 
=\hat v_{ k }(t_n) &
-2i\lambda\sum\limits_{k_1+k_2+k_3=k} \hat{ \bar v}_{k_1}(t_n)\hat v_{k_2}(t_n)\hat v_{k_3}(t_n) \int_0^\tau\frac{k_1+k_2}{k} \fe^{i(t_n+s)\phi}\,ds  \label{v-2-2}\\
    &  +i\lambda\sum\limits_{k_1+k_2+k_3=k}\hat{ \bar v}_{k_1}(t_n)\hat v_{k_2}(t_n)\hat v_{k_3}(t_n) \int_0^\tau\frac{k_1}{k} \fe^{i(t_n+s)\phi}\,ds    \label{v-2-3} \\
    & +\>\hat{\mathcal R}_{1,k} .
\end{align}
\end{subequations}

We furthermore truncate \eqref{v-2-2} to the frequency domain $|k_1+k_3|\le N$, i.e.,  
\begin{align}
  \eqref{v-2-2}= \hat v_{ k }(t_n)
  &-2i \lambda \sum\limits_{{\substack{k_1+k_2+k_3=k\\   |k_1+k_3|\le N}}} \bigg(\int_0^\tau\frac{k_1+k_2}{k} \fe^{i(t_n+s)\phi}\,ds\bigg)
     \>\hat{ \bar v}_{k_1} (t_n) \hat v_{k_2}(t_n) \hat v_{k_3}(t_n) + \hat{\mathcal R}_{4,k}\label{v-2-2-1},
\end{align}
with 
$$
\hat{\mathcal R}_{4,k} = 
\left\{
\begin{aligned}
& -2i \lambda \sum\limits_{{\substack{k_1+k_2+k_3=k\\   |k_1+k_3|> N}}} \bigg(\int_0^\tau\frac{k_1+k_2}{k} \fe^{i(t_n+s)\phi}\,ds\bigg)
     \>\hat{ \bar v}_{k_1} (t_n) \hat v_{k_2}(t_n) \hat v_{k_3}(t_n) 
     &&\mbox{for}\,\,\, 0\neq |k|\le N ,\\
&0 &&\mbox{otherwise} .
\end{aligned}
\right.
$$
The corresponding function $\mathcal R_4$ with Fourier coefficients $\hat{\mathcal R}_{4,k}$ satisfies that 
$$
\mathcal R_4\in \tau \,\mathcal T_3\left(\frac{k_1+k_2}{k}1_{0\neq |k|\le N} 1_{|k_1+k_3|>N};v(t_n)\right).
$$
By Lemma \ref{lem:Rj} (iii) in the next subsection and symmetry between $k_2$ and $k_3$, and we have 
\begin{align}\label{est:R4}
\big\|\mathcal R_4\big\|_{H^s}\lesssim \tau N^{-1+s}\|v\|_{L^\infty_tH^1_x}^3 
\quad\mbox{for}\,\,\, s\in[0,1].
\end{align}
Since $k_1+k_2+k_3=k$, it is straightforward to verify that 
$\phi=2(k_1+k_2)(k_1+k_3).$ As a result, if $k_1+k_3\neq0$ then 
\begin{align}\label{relations-1}
\int_0^\tau\frac{k_1+k_2}{k} \fe^{i(t_n+s)\phi}\,ds
  =\frac{1}{2ik(k_1+k_3)}
     \big(\fe^{it_{n+1}\phi}-\fe^{it_n\phi}\big);
\end{align}
If $k_1 + k_3=0$ then $\phi=0$ and $k=k_2$, and therefore
\begin{align}\label{relations-2}
\int_0^\tau\frac{k_1+k_2}{k} \fe^{i(t_n+s)\phi}\,ds
  =\tau\Big(\frac{k_1}{k}+1\Big).
\end{align}
Substituting the two relations \eqref{relations-1}--\eqref{relations-2} into \eqref{v-2-2-1}, we obtain 
\begin{align*}
\eqref{v-2-2}
   =\hat v_{ k }(t_n) 
   & -\lambda\sum\limits_{{\substack{k_1+k_2+k_3=k\\ 0\neq |k_1+k_3|\le N}}} \frac{1}{k(k_1+k_3)}
     \big(\fe^{it_{n+1}\phi}-\fe^{it_n\phi}\big)
         \>\hat{ \bar v}_{k_1}(t_n)  \hat v_{k_2}(t_n) \hat v_{k_3}(t_n)  \\
   & -2i\lambda \tau\sum\limits_{k_1+k_3= 0} \Big(\frac{k_1}{k}+1\Big) 
         \>\hat{ \bar v}_{k_1}(t_n)  \hat v_{k}(t_n)  \hat v_{k_3}(t_n)  + \hat{\mathcal R}_{4,k} .
\end{align*}
Then we apply the mass and momentum conservations in \eqref{mass-v}--\eqref{momentum-v}, which imply that 
\begin{align*}
 -2i\lambda\tau\sum\limits_{k_1+k_3= 0} \Big(\frac{k_1}{k}+1\Big)
         \>\hat{ \bar v}_{k_1}(t_n) \hat v_{k}(t_n) \hat v_{k_3}(t_n) 
=&-2i\lambda \tau P\>  (ik)^{-1}\hat v_{k}(t_n)   -2i\lambda\tau M  \>\hat v_{k}(t_n)  .
\end{align*}
Therefore, 
\begin{align}\label{est:v-2-2}
\eqref{v-2-2}
   =& \hat v_{ k }(t_n) - \lambda\sum\limits_{{\substack{k_1+k_2+k_3=k\\0\ne |k_1+k_3| \le N}}} \frac{1}{k(k_1+k_3)}
     \big(\fe^{it_{n+1}\phi}-\fe^{it_n\phi}\big)
         \>\hat{ \bar v}_{k_1}(t_n)\hat v_{k_2}(t_n)\hat v_{k_3}(t_n) \notag \\
   & -2i\lambda\tau P   \> (ik)^{-1}\hat v_{k}(t_n) -2i\lambda\tau M  \>\hat v_{k}(t_n) + \hat{\mathcal R}_{4,k} \notag \\[5pt] 
   =
   & e^{-2i\lambda\tau P   \> (ik)^{-1}  -2i\lambda\tau M}  \>\hat v_{k}(t_n) \notag \\  
   & - \lambda\sum\limits_{{\substack{k_1+k_2+k_3=k\\0\ne |k_1+k_3| \le N}}} \frac{1}{k(k_1+k_3)}
     \big(\fe^{it_{n+1}\phi}-\fe^{it_n\phi}\big)
         \>\hat{ \bar v}_{k_1}(t_n)\hat v_{k_2}(t_n)\hat v_{k_3}(t_n) \notag \\
   & + \hat{\mathcal R}_{4,k} + \hat{\mathcal R}_{4,k}^* , 
\end{align}
where 
\begin{align*}
\hat{\mathcal R}_{4,k}^* 
&= 
\left\{
\begin{aligned}
&\big(1-2i\lambda\tau P (ik)^{-1} 
- 2i\lambda\tau M  - e^{-2i\lambda\tau P (ik)^{-1} - 2i\lambda\tau M} \big) \hat v_{k}(t_n)
&&\mbox{for}\,\,\, 0\neq |k|\le N,\\[3pt] 
&0 &&\mbox{otherwise} . 
\end{aligned}
\right.
\end{align*}
From this expression we see that the function $\mathcal{R}_4^*$ with Fourier coefficients $\hat{\mathcal R}_{4,k}^*$ satisfies that 
\begin{align}
\mathcal{R}_4^* 
\in \tau^2 T_1(1;v(t_n)) . 
\end{align}

Note that for $k_1+k_2+k_3=k$ the following equalities hold:
\begin{subequations}
\begin{align}
\phi(k,k_1,k_2,k_3) &=2kk_1+2k_2 k_3 , \label{phase-1}\\
2kk_1 &=k^2+k_1^2-(k_2+k_3)^2 , \\
2k_2k_3 &=(k_2+k_3)^2-k_2^2-k_3^2.\label{phase-2}
\end{align}
\end{subequations}
By using these relations, we have 
$$
e^{i(t_n+s)\phi} 
= e^{it_n \phi} e^{2iskk_1} e^{2isk_2k_3}
= e^{it_n \phi} 
[e^{2iskk_1} + (e^{2isk_2k_3} - 1 ) 
+ (e^{2iskk_1} - 1)(e^{2isk_2k_3} - 1) ] , 
$$
and therefore \eqref{v-2-3} can be decomposed into the following three terms:
\begin{subequations}\label{v-2-45}
\renewcommand{\theequation}
{\theparentequation-\arabic{equation}}
\begin{align}
    \eqref{v-2-3} 
  =&  i\lambda\sum\limits_{{\substack{k_1+k_2+k_3=k \\ |k_2+k_3| \le N}}} 
           \bigg(\int_0^\tau \frac{k_1}{k} \fe^{it_n\phi} \fe^{2iskk_1}\,ds\bigg) 
         \hat{ \bar v}_{k_1}(t_n)\hat v_{k_2}(t_n)\hat v_{k_3}(t_n)  \label{v-2-5}\\
    &+ i\lambda\sum\limits_{{\substack{k_1+k_2+k_3=k\\ |k_2+k_3| \le N}}} 
    \bigg( \int_0^\tau \frac{k_1}{k} \fe^{it_n\phi} \big(\fe^{2isk_2k_3}-1\big)\,ds \bigg) \hat{ \bar v}_{k_1}(t_n)\hat v_{k_2}(t_n)\hat v_{k_3}(t_n) 
          \label{v-2-6}\\
     &+\widehat{\mathcal{R}}_{5,k} +\widehat{\mathcal{R}}_{5,k}^* 
            , \label{v-2-4} 
\end{align}
\end{subequations}
where 
\begin{align}
\widehat{\mathcal{R}}_{5,k} 
&= i\lambda\sum\limits_{{\substack{k_1+k_2+k_3=k\\ |k_2+k_3| \le N}}} 
    \bigg(\int_0^\tau \frac{k_1}{k}
      \fe^{it_n\phi}  \big(\fe^{2iskk_1}-1\big)\big(\fe^{2isk_2k_3}-1\big)\,ds \bigg)
    \hat{ \bar v}_{k_1}(t_n)\hat v_{k_2}(t_n)\hat v_{k_3}(t_n)  ,
    \label{def-R5k} \\ 
\widehat{\mathcal{R}}_{5,k}^* 
&= i\lambda\sum\limits_{{\substack{k_1+k_2+k_3=k\\ |k_2+k_3| > N}}}
\bigg( \int_0^\tau\frac{k_1}{k} \fe^{i(t_n+s)\phi}\,ds\bigg)
\hat{ \bar v}_{k_1}(t_n)\hat v_{k_2}(t_n)\hat v_{k_3}(t_n)  ,
\label{def-R6k} 
\end{align}
for $0\neq |k|\le N$, and $\widehat{\mathcal{R}}_{5,k} =\widehat{\mathcal{R}}_{5,k}^* =0$ for $k=0$ and $|k|>N$. Lemma \ref{lem:R5} of the next subsection implies that 
\begin{subequations}\label{est:R5}
\begin{align}
&\big\|\mathcal R_5\big\|_{H^s}\lesssim \tau^\frac32 \|v\|_{L^\infty_tH^1_x}^3 \quad\mbox{for}\,\,\, s\in (\mbox{$\frac12$},1), \\
&\big\|\mathcal R_5\big\|_{L^2}\lesssim \tau^2\sqrt{\ln\tau^{-1}} \|v\|_{L^\infty_tH^1_x}^3.
\end{align}
\end{subequations}
Obviously, 
$$
\mathcal R_5^*\in \tau  \mathcal T_3\left(\sigma;v(t_n)\right) 
\quad\mbox{with some}\,\,\, |\sigma(k,k_1,k_2,k_3)|\le |k|^{-1}|k_1|1_{0\ne |k|\le N} 1_{|k_2+k_3|>N} .
$$
By Lemma \ref{lem:Rj} (iii) and symmetry, and we have that for any $s\in[0,1]$, 
\begin{align}\label{est:R6}
\big\|\mathcal R_5^*\big\|_{H^s}\lesssim \tau N^{-1+s} \|v\|_{L^\infty_tH^1_x}^3 .
\end{align}

Note that 
\begin{align}
   \eqref{v-2-5}
       &=\sum\limits_{{\substack{k_1+k_2+k_3=k\\ |k_2+k_3| \le N}}}  \frac{\lambda}{2k^2}
              \fe^{it_n\phi}\big(\fe^{i\tau (k^2+k_1^2-(k_2+k_3)^2)}-1\big)
              \hat{ \bar v}_{k_1}(t_n)\hat v_{k_2}(t_n)\hat v_{k_3}(t_n)  ,  \label{v-2-5-est} \\[5pt]
\eqref{v-2-6}
&=
i\lambda\sum\limits_{{\substack{k_1+k_2+k_3=k\\   k_2\ne 0, k_3\ne 0\\ |k_2+k_3| \le N}}} \bigg(\int_0^\tau \frac{k_1}{k}
         \fe^{it_n\phi} \big(\fe^{2isk_2k_3}-1\big)\,ds\bigg) 
         \hat{ \bar v}_{k_1}(t_n)\hat v_{k_2}(t_n)\hat v_{k_3}(t_n) \notag\\
&=
\lambda\sum\limits_{{\substack{k_1+k_2+k_3=k\\   k_2\ne 0, k_3\ne 0\\ |k_2+k_3| \le N}}} \frac{k_1}{2kk_2k_3}
         \fe^{it_n\phi} \big(\fe^{2i\tau k_2k_3}-1\big)
           \hat{ \bar v}_{k_1}(t_n)\hat v_{k_2}(t_n)\hat v_{k_3}(t_n) \notag\\
           &\quad\,
           -i\lambda\tau \sum\limits_{{\substack{k_1+k_2+k_3=k\\   k_2\ne 0, k_3\ne 0\\ |k_2+k_3| \le N}}} \frac{k_1}{k}
         \fe^{it_n\phi}  
           \hat{ \bar v}_{k_1}(t_n)\hat v_{k_2}(t_n)\hat v_{k_3}(t_n) \notag\\
     &  =
    \lambda\sum\limits_{{\substack{k_1+k_2+k_3=k\\   k_2\ne 0, k_3\ne 0\\ |k_2+k_3| \le N}}} \frac{k_1}{2kk_2k_3}
         \fe^{it_n\phi} \Big(\fe^{i\tau \big((k_2+k_3)^2-k_2^2-k_3^2\big)}-1\Big)
           \hat{ \bar v}_{k_1}(t_n)\hat v_{k_2}(t_n)\hat v_{k_3}(t_n)
      \notag   \\   
     &\quad\,  -i\lambda\tau \sum\limits_{{\substack{k_1+k_2+k_3=k\\ |k_2+k_3| \le N}}} \frac{k_1}{k}
         \fe^{it_n\phi}  
           \hat{ \bar v}_{k_1}(t_n)\hat v_{k_2}(t_n)\hat v_{k_3}(t_n) \notag \\
     &\quad\, 
           +2i\lambda \tau \sum\limits_{{\substack{k_1+k_2=k\\ |k_2| \le N}}} \frac{k_1}{k}
         \fe^{it_n\phi}  
           \>\hat{ \bar v}_{k_1}(t_n)\hat v_{k_2}(t_n)\hat v_{0}(t_n) \notag \\
   &\quad\,  -  i\lambda \tau  
         \fe^{it_n\phi}  
           \>\hat{ \bar v}_{k}(t_n)\hat v_{0}(t_n)\hat v_{0}(t_n) .  \label{Eq:v-2-6} 
\end{align}
 
  
Substituting \eqref{v-2-5-est}--\eqref{Eq:v-2-6} into \eqref{v-2-45}, 
and then substituting \eqref{est:v-2-2} and \eqref{v-2-45} into \eqref{v-2-0}, we obtain  
\begin{align}\label{est:v-n0}
\hat v_k(t_{n+1})
=&e^{-2i\lambda\tau P   \> (ik)^{-1}  -2i\lambda\tau M}  \>\hat v_{k}(t_n)  \notag\\
&
    +\lambda \sum\limits_{{\substack{k_1+k_2+k_3=k\\ 0\ne |k_1+k_3| \le N}}} \frac{1}{ik(ik_1+ ik_3)}
     \big(\fe^{it_{n+1}\phi}-\fe^{it_n\phi}\big) 
         \>\hat{ \bar v}_{k_1}(t_n)\hat v_{k_2}(t_n)\hat v_{k_3}(t_n)\notag\\
   &
         - \lambda \sum\limits_{{\substack{k_1+k_2+k_3=k\\ |k_2+k_3| \le N}}}  \frac{1}{2(ik)^2}
              \fe^{it_n\phi} \big(\fe^{i\tau (k^2+k_1^2-(k_2+k_3)^2)}-1\big)
               \>\hat{ \bar v}_{k_1}(t_n)\hat v_{k_2}(t_n)\hat v_{k_3}(t_n)\notag\\
       &-  \lambda\sum\limits_{{\substack{k_1+k_2+k_3=k\\ k_2\ne 0, k_3\ne 0\\ |k_2+k_3| \le N}}} \frac{ik_1}{2(ik)( ik_2)( ik_3)}
         \fe^{it_n\phi} \big(\fe^{i\tau ((k_2+k_3)^2-k_2^2-k_3^2)}-1\big) 
           \>\hat{ \bar v}_{k_1}(t_n)\hat v_{k_2}(t_n)\hat v_{k_3}(t_n) 
        \notag  \\  
     & -i\lambda\tau \sum\limits_{{\substack{k_1+k_2+k_3=k\\ |k_2+k_3| \le N}}} \frac{k_1}{k}
         \fe^{it_n\phi}  
           \>\hat{ \bar v}_{k_1}(t_n)\hat v_{k_2}(t_n)\hat v_{k_3}(t_n) \notag\\
     & +2i\lambda \tau \sum\limits_{{\substack{k_1+k_2=k\\ |k_2| \le N}}} \frac{k_1}{k}
         \fe^{it_n\phi}  
           \>\hat{ \bar v}_{k_1}(t_n)\hat v_{k_2}(t_n)\hat v_{0}(t_n)  \notag\\
     & -  i \lambda \tau  
         \fe^{it_n\phi}  
           \>\hat{ \bar v}_{k}(t_n)\hat v_{0}(t_n)\hat v_{0}(t_n)\notag\\[5pt] 
 & +\hat{\mathcal R}_{1,k}+\hat{\mathcal R}_{4,k} +\hat{\mathcal R}_{4,k} ^* +\hat{\mathcal R}_{5,k} +\hat{\mathcal R}_{5,k}^* 
 \quad\mbox{for}\,\,\, k\ne 0 \,\,\,\mbox{and}\,\,\, |k|\le N.
\end{align}
Then substituting \eqref{zeromode-1} and \eqref{est:v-n0} into the expression $v(t_{n+1})= \sum_{k\in\Z} \hat v_k(t_{n+1})e^{ikx}$ yields 
\begin{align}\label{vtn+1-app}
v(t_{n+1})=
&\Phi^n(v(t_n);M,P)+ \mathcal R_{1}+\hat{\mathcal R}_{2,0} +  \hat{\mathcal R}_{2,0}^* +  {\mathcal R}_{3}
+ \mathcal R_4 + \mathcal R_4^* + \mathcal R_5 + \mathcal R_5^* ,
\end{align}
where 
\begin{align}\label{Phi-def}
\Phi^n(f; M,P)
:=\, 
& \fe^{-2i\lambda\tau P\partial_x^{-1}-2i\lambda\tau M}f + (1-\fe^{-2i\lambda\tau M}) \Pi_0 f \notag \\
& 
-i\lambda \tau \Pi_0\big[ \Pi_N\big(\big|\fe^{it_n \partial_{x}^2}f\big|^2\big)\fe^{it_n \partial_{x}^2}f\big]
\notag\\
 & + \lambda\fe^{-it_{n+1}\partial_x^2} \partial_x^{-1}\Pi_N
          \Big[ \big(\fe^{it_{n+1}\partial_x^2} f\big)
              \cdot \partial_x^{-1}\Pi_N
                \big(|\fe^{it_{n+1}\partial_x^2} f|^2 \big)   \Big]\notag\\
 &
                -\lambda \fe^{-it_n\partial_x^2} \partial_x^{-1}\Pi_N
          \Big[ \big(\fe^{it_n\partial_x^2} f \big)
              \cdot \partial_x^{-1} \Pi_N
                \big(|\fe^{it_n\partial_x^2} f |^2 \big)   \Big]\notag\\     
 & -\frac{\lambda}{2} \bigg[  \fe^{-it_{n+1}\partial_x^2} \partial_x^{-2}\Pi_N
             \Big(\big(\fe^{-it_{n+1}\partial_x^2}\bar f\big)
              \cdot \fe^{i\tau\partial_x^2}\Pi_N
                  \big(\fe^{it_n\partial_x^2} f\big)^2\Big)\notag\\
 &
                 \qquad\quad - \fe^{-it_n\partial_x^2} \partial_x^{-2}\Pi_N
          \Big( \fe^{-it_n\partial_x^2} \bar f  \Pi_N
             \big(\fe^{it_n\partial_x^2} f\big)^2\Big) \bigg] \notag \\ 
 & -\frac{\lambda}{2}
 \bigg[ \fe^{-it_n\partial_x^2} \partial_x^{-1}\Pi_N
          \Big( \big(\fe^{-it_n\partial_x^2} \partial_x \bar f\big)
              \cdot \fe^{-i\tau\partial_x^2} \Pi_N
                \big(\fe^{it_{n+1}\partial_x^2}\partial_x^{-1} f \big)^2   \Big) \notag\\
 &
                \qquad\quad - \fe^{-it_n\partial_x^2} \partial_x^{-1}\Pi_N
          \Big( \big(\fe^{-it_n\partial_x^2} \partial_x \bar f\big)
              \cdot \Pi_N
              \big(\fe^{it_n\partial_x^2}\partial_x^{-1} f \big)^2   \Big) \bigg] \notag\\
&
-i\lambda\tau \fe^{-it_n\partial_x^2}\partial_x^{-1} \Pi_N
\Big( \fe^{-it_n\partial_x^2} \partial_x\bar f \Pi_N\big(\fe^{it_n\partial_x^2} f\big)^2\Big) \notag \\ 
& +2i\lambda \tau \Pi_0(f) \fe^{-it_n\partial_x^2} \partial_x^{-1}  \left(\fe^{-it_n\partial_x^2}\partial_x\bar f\>\fe^{it_n \partial_{x}^2}f\right) \notag\\
&
-i\lambda\tau (\Pi_0 f)^2 \Pi_{\ne 0}\big(\fe^{-it_n\partial_x^2}\bar f\big) .
\end{align}

The numerical scheme can be defined by dropping the defect terms $\mathcal{R}_j$ and $\mathcal{R}_j^*$ in \eqref{vtn+1-app} and replacing the numbers $M$ and $P$ by their approximations $M_N$ and $P_N$ defined in \eqref{M_N+P_N}, respectively. Namely, for given $v_n\in S_N$ compute $v_{n+1}\in S_N$ by 
\begin{align}\label{NuSo-NLS}
v^{n+1}=\Phi^n(v^n;M_N,P_N), \quad n=0,1\ldots, L - 1;\quad\mbox{with}\,\,\, v^0=u^0.
\end{align}
Then, replacing $v^n$ and $v^{n+1}$ by $\fe^{-it_n\partial_x^2}u^n$ and $\fe^{-it_{n+1}\partial_x^2}u^n$ in \eqref{NuSo-NLS}, we obtain the numerical scheme \eqref{numerical}--\eqref{psi}.


\subsection{Technical lemmas for analysing the consistency errors}
\label{section:consistency}

In this subsection, we present two technical lemmas, which are used in estimating the defect terms $\mathcal R_j$ and $\mathcal R_j^*$ in the previous subsection. 
\begin{lemma}\label{lem:Rj}
For any given $v\in H^1(\T)$ and $s\in [0,1]$, the following results hold. 
\begin{itemize}
\item[(i)]
Let $m\geq1, N\in \Z^+$. Then, for any $f\in \mathcal T_m(1;v)$ and any $g\in \mathcal T_m(1_{>N};v)$,  
\begin{align*}
\big\|f\big\|_{H^1}
&\lesssim \|v\|_{H^1}^m; \\
\big\|g\big\|_{H^s}
&\lesssim N^{-1+s} \|v\|_{H^1}^m.
\end{align*}
\item[(ii)]
For any $f\in \mathcal T_3(k_2k_3;v)$ there holds 
$$
 |\Pi_0 f | \lesssim \|v\|_{H^1}^3.
$$
\item[(iii)]  
Let $N\in \Z^+$, $N\ge 10$ and $f \in \mathcal T_3(\sigma;v)$. If 
$$
| \sigma( k, k _1,k_2, k _3)| \lesssim |k|^{-1}|k_j|\>1_{0\ne |k|\le N} \, 1_{|k_1+k_2|>N}, 
\quad\mbox{for some $j\in \{1,2,3\}$} , 
$$
then 
$$
 \left\| f \right\|_{H^s}\lesssim N^{-1+s} \|v\|_{H^1}^3. 
$$
\end{itemize}
\end{lemma}
\begin{proof}
Without loss of generality, we can assume that $\hat{v}_{k_j}, j=1,\cdots,m$ are positive for any $t\in [0,T]$. Otherwise we replace $\hat{v}_{k_j}$ by $|\hat{v}_{k_j}|$ as we did in the proof of Lemma \ref{lem:kato-Ponce-1}.

(i) By the definition of $\mathcal T_m(\sigma;v)$ in \eqref{def:TmM}, $f\in \mathcal T_m(1;v)$ implies that 
\begin{align*}
|\hat f_k|\lesssim 
\sum\limits_{k _1+\cdots+ k _m = k}
\>\hat{v}_{ k _1}\cdots \hat v_{ k _m} \sim \mathcal{F}_k[v^m] .
\end{align*}
By Plancherel's identity and Lemma \ref{lem:kato-Ponce} (i), we obtain that 
\begin{align*}
\|f\|_{H^1}
  \lesssim \|v^m\|_{H^1}
\lesssim \|v\|_{H^1}^m.
\end{align*}
For $g \in \mathcal T_m(1_{>N};v)$, we use the inequality $\|g\|_{H^s}
  \lesssim
 N^{-1+s} \|g\|_{H^1}$ together with the inequality above, which implies that 
$\|g\|_{H^1}
\lesssim \|v\|_{H^1}^m$. This yields the desired inequality for $g$, i.e., 
\begin{align*}
\|g\|_{H^s}
\lesssim  N^{-1+s} \|v\|_{H^1}^m .
\end{align*}

(ii) For any $f\in \mathcal T_3(k_2k_3;v)$ we have that 
\begin{align*}
|\Pi_0f|
\lesssim &
\sum\limits_{k _1+k_2+ k _3=0}
\>\hat{v}_{ k _1}(t) \>|k_2|\hat{v}_{ k _2}(t) \>|k_3|\hat{v}_{ k _3}(t) \\
\lesssim & \sum\limits_{k _1+k_1'=0} \hat{v}_{ k _1}(t) \mathcal{F}_{k_1'}[ (|\nabla |v(t))^2 ] \\
\lesssim & 
\int_\T v \left(|\nabla |v\right)^2\,dx 
\lesssim \|v\|_{L^\infty}  \| |\nabla| v\|_{L^2}^2 
\lesssim \|v\|_{H^1}^3.
\end{align*}

(iii) We only consider the case when $j=1$, since the other cases can be treated in the same way.  
Since the Fourier coefficients of $J^sf$ satisfies 
\begin{align*}
\mathcal{F}_k[J^sf]
= \langle k\rangle^{s} \hat f_k  
\lesssim & 
\sum\limits_{{\substack{k_1+k_2+k_3 = k \\ |k_1+k_2|> N}}} 1_{0\ne |k|\le N}
\langle k\rangle^{-1+s} \langle k_1\rangle
\>\hat{v}_{ k _1}(t) \hat{v}_{ k _2}(t)\hat v_{ k_3}(t) \\
\lesssim & 
N^{-1+s} 
\sum\limits_{{\substack{k_1+k_2+k_3 = k \\ |k_1+k_2|> N}}} 
\langle k\rangle^{-1}
\langle k_1+k_2 \rangle\langle k_1\rangle \hat{v}_{ k _1}(t) \hat{v}_{ k _2}(t)\hat v_{ k_3}(t) \\
\lesssim &
N^{-1+s}\mathcal{F}_k[ J^{-1} (v J(vJv)) ] ,
\end{align*}
it follows from Lemma \ref{lem:kato-Ponce-1} (ii) that 
\begin{align*}
\| J^sf\|_{L^2} 
\lesssim &
N^{-1+s}  \|vJv\|_{L^2} \|v\|_{H^1}  
\lesssim 
N^{-1+s}  \|v\|_{L^\infty} \|Jv\|_{L^2} \|v\|_{H^1} 
\lesssim 
N^{-1+s}  \|v\|_{H^1}^3 . 
\end{align*}

This proves the desired results in Lemma \ref{lem:Rj}. 

\end{proof}

\begin{lemma}\label{lem:R5}
If $v\in L^\infty(0,T;H^1(\T))$ then
\begin{align}
\big\|\mathcal R_5\big\|_{L^2}\lesssim \tau^2\sqrt{\ln\tau^{-1}} \|v\|_{L^\infty(0,T;H^1)}^3.\label{R5-L2}
\end{align}
Moreover, for any $s\in (\frac12, 1)$, 
\begin{align}
\big\|\mathcal R_5\big\|_{H^s}\lesssim \tau^\frac32 \|v\|_{L^\infty(0,T;H^1)}^3.\label{R5-s}
\end{align}
\end{lemma}
\begin{proof}
For $k_1+k_2+k_3=k$ and $|k_2|\ge |k_3|$ we claim that the following inequality holds:  
\begin{align}
\left|\frac{k_1}{k}
       \big(\fe^{2iskk_1}-1\big)\big(\fe^{2isk_2k_3}-1\big)\right|
      \lesssim 
     \tau |k|^{-\alpha}  |k_1||k_2| |k_3|^\alpha
     \quad\forall\, s\in [0,\tau],\,\,\,\forall\, \alpha\in [0,1].
  \label{est:multiplier}
\end{align}
In order to prove \eqref{est:multiplier}, we consider the following two cases: $|k|\ge |k_3|$ and $|k|< |k_3|$. 

{\sc Case} 1: $|k|\ge |k_3|$. In this case, we use the following inequalities: 
$$
\big|\fe^{2iskk_1}-1\big|\le 2 
\quad\mbox{and}\quad \big|\fe^{2isk_2k_3}-1\big|\le 2\tau |k_2||k_3| , 
$$
it follows that 
\begin{align*}
\left|\frac{k_1}{k}
       \big(\fe^{2iskk_1}-1\big)\big(\fe^{2isk_2k_3}-1\big)\right|
       \le 4\tau |k|^{-1}|k_1| |k_2||k_3|
\lesssim \tau |k|^{-\alpha}|k_1| |k_2||k_3|^\alpha.
\end{align*}

{\sc Case} 2: $|k|< |k_3|$. In this case $k_1+k_2+k_3=k$ and $|k_2|\ge |k_3|$ imply 
$$
|k_1|\le |k_2|+|k_3|+|k|\lesssim |k_2| .
$$ 
We use the following inequalities:
$$
\big|\fe^{2iskk_1}-1\big|\le 2\tau |k||k_1|
\quad\mbox{and}\quad  
\big|\fe^{2isk_2k_3}-1\big|\le 2.
$$
Then we obtain 
\begin{align*}
\left|\frac{k_1}{k}
       \big(\fe^{2iskk_1}-1\big)\big(\fe^{2isk_2k_3}-1\big)\right|
       \le 4\tau |k_1|^2.
\end{align*}
Since $|k_1|\lesssim |k_2|$, it follows that 
$$
|k_1|^2\lesssim |k_1| |k_2|\lesssim   |k_1| |k_2| \left(\frac{|k_3|}{|k|}\right)^\alpha.
$$
This proves \eqref{est:multiplier}. 

By using the symmetry between $k_2$ and $k_3$ in the expression of $|\hat{\mathcal R}_{5,k}|$ in \eqref{def-R5k}  and applying \eqref{est:multiplier} with $\alpha=1$ in the case $|k|\ge |k_3|$ and $\alpha=0$ in the case $|k|<|k_3|$, we obtain for any $k\ne 0$, 
\begin{align}
|\hat{\mathcal R}_{5,k}| 
\lesssim &\, \tau^2 \sum\limits_{{\substack{k_1+k_2+k_3=k\\ |k_2|\ge |k_3|,  |k|\ge |k_3|}}} 
 |k|^{-1}  |k_1||k_2| |k_3| |\hat{ \bar v}_{k_1}(t_n)||\hat v_{k_2}(t_n)||\hat v_{k_3}(t_n)| \notag\\ 
& + \tau^2 \sum\limits_{{\substack{k_1+k_2+k_3=k\\ |k_2|\ge |k_3|,  |k|< |k_3|}}} 
  |k_1||k_2|  |\hat{ \bar v}_{k_1}(t_n)||\hat v_{k_2}(t_n)||\hat v_{k_3}(t_n)|.\label{R5-k}
\end{align}
Without loss of generality, we may assume that $\hat{ \bar v}_{k_1}(t_n),\hat v_{k_2}(t_n)$ and $\hat v_{k_3}(t_n)$ are nonnegative. Otherwise we replace them by their absolute values as we did in the proof of Lemma \ref{lem:kato-Ponce-1}. 

By the duality between $L^2(\T)$ and itself, it is sufficient to prove the following result to obtain \eqref{R5-L2}: 
\begin{align}\label{desired-R5}
|\langle \mathcal R_5,f\rangle|\lesssim \tau^2 \sqrt{\ln (\tau^{-1})} \|v\|_{L^\infty(0,T;H^1)}^3 \|f\|_{L^2}
\quad\forall\, f\in L^2(\T) . 
\end{align}
From the definition below \eqref{def-R6k} we see that $\mathcal{R}_{5,0}=0$. As a result, we have  
\begin{align}
|\langle  \mathcal R_5,f\rangle| 
\lesssim & \sum\limits_{k\ne 0} |\hat{\mathcal R}_{5,k}| \>|\hat f_k| 
\lesssim \sum\limits_{|k|> \tau^{-1}} |\hat{\mathcal R}_{5,k}| \>|\hat f_k|
+\sum\limits_{0\ne |k|\le \tau^{-1}} |\hat{\mathcal R}_{5,k}| \>|\hat f_k| . \label{R5-1-2}
\end{align}
From the expression of $\mathcal{R}_{5,k}$ in \eqref{def-R5k} we see that for $|k|>\tau^{-1}$ there holds 
$$
|\hat{\mathcal R}_{5,k}|\le \tau^2 \sum\limits_{k_1+k_2+k_3=k} |k_1| \hat{ \bar v}_{k_1}(t_n)\hat v_{k_2}(t_n)\hat v_{k_3}(t_n) .
$$
Hence, by the Cauchy--Schwartz inequality and Plancherel's identity, we have 
\begin{align}
 \sum\limits_{|k|> \tau^{-1}} |\hat{\mathcal R}_{5,k}| \>|\hat f_k|
\le & \tau^2 \sum\limits_{k}  \sum\limits_{k_1+k_2+k_3=k} |k_1| \hat{ \bar v}_{k_1}(t_n)\hat v_{k_2}(t_n)\hat v_{k_3}(t_n)\>|\hat f_k|\notag\\
= & \tau^2  \sum\limits_{k_2,k_3} \sum\limits_{k}  |k-k_2-k_3| \hat{ \bar v}_{k-k_2-k_3}(t_n)\hat v_{k_2}(t_n)\hat v_{k_3}(t_n)\>|\hat f_k|\notag\\
\lesssim & \tau^2\|(\hat f_k)_{k\in\Z}\|_{l^2} \|(k_1 \hat{\bar v}_{k_1}(t_n) )_{k_1\in\Z}\|_{l^2} \|(\hat v_{k_2}(t_n))_{k_2\in\Z}\|_{l^1}\|(\hat v_{k_3}(t_n))_{k_3\in\Z}\|_{l^1}\notag\\
\lesssim & \tau^2 \|f\|_{L^2} \|v\|_{H^1}^3 , \label{R5-HF}
\end{align}
where the last inequality uses the following result: 
$$
\|( \hat v_{k_2}(t_n) )_{k_2\in\Z}\|_{l^1}
\lesssim
\|(\langle k_2\rangle^{-1})_{k_2\in\Z}\|_{l^2} \|(\langle k_2\rangle\hat v_{k_2}(t_n))_{k_2\in\Z}\|_{l^2} 
\lesssim 
\|v\|_{H^1} . 
$$
The second term in \eqref{R5-1-2} can be estimated by using \eqref{R5-k}, i.e., 
\begin{align}
&\hspace{-15pt} \sum\limits_{0\ne |k|\le \tau^{-1}} |\hat{\mathcal R}_{5,k}| \>|\hat f_k| \\
\lesssim & \tau^2\sum\limits_{0\ne |k|\le \tau^{-1}} \sum\limits_{{\substack{k_1+k_2+k_3=k\\ |k_2|\ge |k_3|,  |k|\ge |k_3|}}} 
 |k|^{-1}  |k_1||k_2| |k_3| |\hat f_k| \hat{ \bar v}_{k_1}(t_n)\hat v_{k_2}(t_n)\hat v_{k_3}(t_n) \notag \\
 & + \tau^2 \sum\limits_{0\ne |k|\le  \tau^{-1}} \sum\limits_{{\substack{k_1+k_2+k_3=k\\ |k_2|\ge |k_3|,  |k|< |k_3|}}} 
  |k_1||k_2| |\hat f_k| \hat{ \bar v}_{k_1}(t_n)\hat v_{k_2}(t_n)\hat v_{k_3}(t_n) \notag\\
\lesssim & \tau^2\sum_{0\ne |k|\le \tau^{-1}}  \sum_{|k_3|\le |k|}  \sum_{k_1} 
 |k|^{-1}   |\hat f_k| |k_1|\hat{ \bar v}_{k_1} (t_n)|k-k_1-k_3|\hat v_{k-k_1-k_3}(t_n) |k_3| \hat v_{k_3}  (t_n)\notag\\
 & + \tau^2\sum_{0\ne |k|\le \tau^{-1}}  \sum_{|k_3|> |k|}  \sum_{k_1} 
   |\hat f_k| |k_1|\hat{ \bar v}_{k_1}(t_n)  |k-k_1-k_3| \hat v_{k-k_1-k_3}(t_n)\hat v_{k_3} (t_n)\notag\\
\lesssim & \tau^2\|(k_1\hat v_{k_1}(t_n))_{k_1\in\Z}\|_{l^2}
\|(k_2\hat v_{k_2}(t_n) )_{k_2\in\Z}\|_{l^2} \sum_{0\ne |k|\le \tau^{-1}}  |k|^{-1} |\hat f_k|  \sum_{|k_3|\le |k|} 
|k_3| \hat v_{k_3}(t_n)  \notag\\
&
+ \tau^2\|(k_1\hat v_{k_1}(t_n) )_{k_1\in\Z}\|_{l^2}
\|(k_2\hat v_{k_2}(t_n) )_{k_2\in\Z}\|_{l^2} \sum_{0\ne |k|\le \tau^{-1}}  |\hat f_k|  \sum_{|k_3|> |k|} \hat v_{k_3} (t_n) \notag\\
\lesssim & \tau^2\|v(t_n) \|_{H^1}^2 
\sum_{0\ne |k|\le \tau^{-1}} |k|^{-\frac12} |\hat f_k|  
\|(k_3\hat v_{k_3}(t_n) )_{k_3\in\Z}\|_{l^2}  \notag\\
&
+ \tau^2\|v(t_n) \|_{H^1}^2 \sum_{0\ne |k|\le \tau^{-1}}  |\hat f_k|  \| (\langle k_3 \rangle^{-1})_{|k_3|> k}\|_{l^2} \|( \langle k_3 \rangle\hat v_{k_3}(t_n) )_{|k_3|> k}\|_{l^2}  \notag\\ 
\lesssim & \tau^2\|v(t_n) \|_{H^1}^2 
\sum_{0\ne |k|\le \tau^{-1}} |k|^{-\frac12} |\hat f_k|  
\|(k_3\hat v_{k_3}(t_n) )_{k_3\in\Z}\|_{l^2}  \notag\\
\lesssim & \tau^2\|v(t_n) \|_{H^1}^3 
\|(|k|^{-\frac12})_{0\ne |k|\le \tau^{-1}}\|_{l^2}  \|(\hat f_k)_{|k|\le \tau^{-1}}\|_{l^2} \notag \\
\lesssim & \tau^2\|v(t_n) \|_{H^1}^3 \sqrt{\ln (\tau^{-1})} \|f\|_{L^2}   .
\label{R5-HF2}
\end{align}
Substituting \eqref{R5-HF}--\eqref{R5-HF2} into \eqref{R5-1-2} yields \eqref{desired-R5}, which implies the desired result in \eqref{R5-L2}. 

It remains to prove \eqref{R5-s}. To this end, we use the following inequalities: 
$$
|\fe^{2iskk_1}-1|\le 2\quad\mbox{and}\quad
|\fe^{2isk_2k_3}-1|\lesssim s^{\frac12}||k_2|^\frac12 |k_3|^\frac12 ,
$$
which imply that 
\begin{align*}
\left|\frac{k_1}{k}
       \big(\fe^{2iskk_1}-1\big)\big(\fe^{2isk_2k_3}-1\big)\right|
      \lesssim 
     \tau^\frac12 |k|^{-1}  |k_1||k_2|^\frac12 |k_3|^\frac12 \quad\forall\, s\in[0,\tau].
\end{align*}
By substituting this into the expression of $\hat{\mathcal{R}}_{5,k}$ in \eqref{def-R5k}, and using Plancherel's identity, we obtain   
\begin{align*}
\big\|\mathcal R_5\big\|_{H^s}\lesssim \tau^\frac32 
\left\||\nabla|^{-1+s}\left(|\nabla|\bar v \> \big(|\nabla|^\frac12 v\big)^2\right)\right\|_{L^2}.
\end{align*}
Then using the Sobolev inequality, we get that for any $s\in (\frac12, 1)$, 
\begin{align*}
\big\|\mathcal R_5\big\|_{H^s}
\lesssim & \tau^\frac32 
\left\||\nabla|\bar v \> \big(|\nabla|^\frac12 v\big)^2\right\|_{L^{\frac{2}{3-2s}}}\\
\lesssim & \tau^\frac32 
\big\||\nabla|\bar v \big\|_{L^2}\> \big\| |\nabla|^\frac12 v\big\|_{L^{\frac{2}{1-s}}}^2
\lesssim\tau^\frac32  \|v\|_{H^1}^3.
\end{align*}
This completes the proof of Lemma \ref{lem:R5}. 
\end{proof}

\section{Proof of Theorem \ref{main:thm1}}\label{section:proof}

The proof of Theorem \ref{main:thm1} is divided into two parts. 
In subsection \ref{section:bdH1}, we present an error estimate for the numerical solution in $H^{s}(\T)$ with $s\in(\frac12,1)$, and then use this result to prove the boundedness of the numerical solution in $H^1(\T)$ uniformly with respect to $\tau$ and $N$. 
In subsection \ref{section:error-L2}, we utilize the $H^1$-boundedness of the numerical solution to prove the desired error estimate in $L^2(\T)$. 

\subsection{Boundedness of the numerical solution in $H^1(\T)$}
\label{section:bdH1}

\begin{lemma}\label{lem:Apriori}
Let $u^0\in H^1(\T)$, and let $u^n_{\tau,N}$, $n=0,1,\dots,L$, be the numerical solution given by \eqref{numerical}--\eqref{psi}. Then there exist positive constants $\tau_s$ and $N_s$ such that for $\tau\in(0,\tau_s]$ and $N\ge N_s$ the following error bound holds:  
\begin{equation}\label{AP-2.30}
\max_{0\le n\le L} \|u(t_n,\cdot)-u^n_{\tau,N}\|_{H^{s}} \lesssim_s \tau^{\frac12}+N^{-1+s} 
  \quad\forall\, s\in(\mbox{$\frac12$},1) , 
\end{equation}
where $\tau_s$ and $N_s$ depend only on $\|u^0\|_{H^1}$, $T$ and $s$. 
\end{lemma}
\begin{proof}
Let $v^n=e^{-it_n\partial_x^2}u^n_{\tau,N}$. Then $v^{n+1}=\Phi^n(v^n;M_N,P_N)$ as shown in \eqref{NuSo-NLS}. By using this identity we have 
\begin{align}
v(t_{n+1})-v^{n+1}=&v(t_{n+1})-\Phi^n(v(t_n);M,P)+ \Phi^n(v(t_n);M,P)-\Phi^n(v^n;M_N,P_N) \notag \\
=&\hspace{-3pt}: \mathcal{L}^n+\Phi^n(v(t_n);M,P)-\Phi^n(v^n;M_N,P_N), 
\label{vtn+1-vn+1}
\end{align}
where 
$$
\mathcal{L}^n=v(t_{n+1})-\Phi^n(v(t_n);M,P)
= \mathcal R_{1}+\hat{\mathcal R}_{2,0} +  \hat{\mathcal R}_{2,0}^* +  \hat{\mathcal R}_{3,0}
+ \mathcal R_4 + \mathcal R_4^* + \mathcal R_5 + \mathcal R_5^* , 
$$
which is shown in \eqref{vtn+1-app}. From \eqref{est:R1}, \eqref{est:R2}, \eqref{est:R3}, \eqref{est:R4}, \eqref{est:R5} and \eqref{est:R6} we see that 
\begin{align}\label{est:LHs}
\big\|\mathcal{L}^n\big\|_{H^{s}} \lesssim \tau^\frac32+\tau N^{-1+s} 
\quad\forall\, s\in[0,1) .
\end{align}
Note that the functional $\Phi^n(f;M,P)$ defined in \eqref{Phi-def} can be rewritten into the following  form:
\begin{align}\label{expr-Phinf}
\Phi^n(f;M,P)=f 
&+ \big(\fe^{-2i\lambda\tau P\partial_x^{-1}-2i\lambda\tau M} - 1+ 2i\lambda\tau P\partial_x^{-1}+2i\lambda\tau M\big) f + (1-\fe^{-2i\lambda\tau M}) \Pi_0 f \notag \\[5pt] 
& -i\lambda \tau \Pi_0\big[ \Pi_N\big(\big|\fe^{it_n \partial_{x}^2} f\big|^2\big)\fe^{it_n \partial_{x}^2} f\big]
\notag\\[5pt] 
 & -2i\lambda \sum\limits_{0\ne |k|\le N}\fe^{ikx}
 \bigg( \sum\limits_{{\substack{k_1+k_2+k_3=k\\ |k_2+k_3|\le N}}} \int_0^\tau\frac{k_1+k_2}{k} \fe^{i(t_n+s)\phi}\,ds \bigg) 
    \>\hat{\bar f}_{k_1}\hat f_{k_2}\hat f_{k_3} \notag \\
&+i\lambda \sum\limits_{0\ne |k|\le N}\fe^{ikx}\sum\limits_{{\substack{k_1+k_2+k_3=k\\ |k_2+k_3|\le N}}} \bigg( \int_0^\tau \frac{k_1}{k} \fe^{it_n\phi} \fe^{2iskk_1}\,ds \bigg)
           \>\hat{\bar f}_{k_1}\hat f_{k_2}\hat f_{k_3} \notag \\
    &+ i\lambda \sum\limits_{0\ne |k|\le N}\fe^{ikx}\sum\limits_{{\substack{k_1+k_2+k_3=k\\ |k_2+k_3|\le N}}} \bigg(\int_0^\tau \frac{k_1}{k} \fe^{it_n\phi} \big(\fe^{2isk_2k_3}-1\big)\,ds \bigg) 
           \>\hat{ \bar f}_{k_1}\hat f_{k_2}\hat f_{k_3}  .
\end{align}
For example, the third line of \eqref{expr-Phinf} comes from \eqref{est:v-2-2}, which can be rewritten back into \eqref {v-2-2-1}. This is how we obtain the third line in the expression above. The other terms are obtained similarly. 

From \eqref{expr-Phinf} we furthermore derive that 
\begin{align}\label{Phivtn-Phivn-1}
\Phi^n(v(t_n);M,P)-\Phi^n(v^n;M_N,P_N)
     =&v(t_n)-v^n+\Phi^n_1+\Phi^n_2+\Phi^n_3+\Phi^n_4+\Phi^n_5,
\end{align}
where
\begin{align*}
\Phi^n_1
= 
&\big(\fe^{-2i\lambda\tau P\partial_x^{-1}-2i\lambda\tau M} - 1+ 2i\lambda\tau P\partial_x^{-1}+2i\lambda\tau M + (1-\fe^{-2i\lambda\tau M})\Pi_0 \big) v(t_n) \notag\\ 
&-\big(\fe^{-2i\lambda\tau P_N\partial_x^{-1}-2i\lambda\tau M_N} - 1+ 2i\lambda\tau P_N\partial_x^{-1}+2i\lambda\tau M_N + (1-\fe^{-2i\lambda\tau M_N})\Pi_0 \big) v^n, \\[5pt]
\Phi^n_2=& -i\lambda\tau \Pi_0\Big(|\fe^{it_n \partial_{x}^2}v(t_n)|^2\fe^{it_n \partial_{x}^2}v(t_n) - |\fe^{it_n \partial_{x}^2}v^n|^2\fe^{it_n \partial_{x}^2}v^n \Big) , \\[5pt]
\Phi^n_3=& -2i\lambda  \hspace{-5pt}  \sum\limits_{0\ne |k|\le N}  \hspace{-5pt} \fe^{ikx} \hspace{-10pt} \sum\limits_{{\substack{k_1+k_2+k_3=k\\ |k_2+k_3|\le N}}}   \hspace{-5pt}  \bigg( \int_0^\tau\frac{k_1+k_2}{k} \fe^{i(t_n+s)\phi}\,ds \bigg) \big( \hat{\bar v}_{k_1}(t_n) \hat{v}_{k_2}(t_n)\hat{v}_{k_3}(t_n) -\hat{\bar v}_{k_1}^n \hat{v}_{k_2}^n\hat{v}_{k_3}^n \big) , \\
 \Phi^n_4=&i\lambda \hspace{-5pt} \sum\limits_{0\ne |k|\le N}\fe^{ikx} \hspace{-10pt}\sum\limits_{{\substack{k_1+k_2+k_3=k\\ |k_2+k_3|\le N}}} \hspace{-5pt} \bigg( \int_0^\tau \frac{k_1}{k} \fe^{it_n\phi} \fe^{2iskk_1}\,ds\bigg) 
 \big( \hat{\bar v}_{k_1}(t_n) \hat{v}_{k_2}(t_n)\hat{v}_{k_3}(t_n) -\hat{\bar v}_{k_1}^n \hat{v}_{k_2}^n\hat{v}_{k_3}^n \big) ,  \\
\Phi^n_5=&i\lambda \hspace{-5pt} \sum\limits_{0\ne |k|\le N} \hspace{-5pt} \fe^{ikx}\hspace{-10pt}\sum\limits_{{\substack{k_1+k_2+k_3=k\\ |k_2+k_3|\le N}}} \hspace{-5pt} \bigg( \int_0^\tau \frac{k_1}{k} \fe^{it_n\phi} \big(\fe^{2isk_2k_3}-1\big)\,ds \bigg)
 \big( \hat{\bar v}_{k_1}(t_n) \hat{v}_{k_2}(t_n)\hat{v}_{k_3}(t_n) -\hat{\bar v}_{k_1}^n \hat{v}_{k_2}^n\hat{v}_{k_3}^n \big) .         
\end{align*}

Note that $P$, $M$, $P_N$ and $M_N$ defined in \eqref{M+P} and \eqref{M_N+P_N} are all bounded numbers, with bounds depending on $\|u^0\|_{H^1}$. In particular, 
\begin{align}\label{M-M_N-error}
|M-M_N| 
&=
\bigg| \frac{1}{2\pi} \int_{\T} (|u^0|^2 -|u^0_{\tau,N}|^2) \,d x \bigg| \notag \\
&\lesssim
\bigg| \frac{1}{2\pi} \int_{\T} 
\Big[ (u^0-u_{\tau,N}^0)\overline{u^0} 
+u_{\tau,N}^0\overline{(u^0-u_{\tau,N}^0)} \Big] \,d x \bigg| \notag \\
&\lesssim 
\|u^0-u_{\tau,N}^0\|_{L^2} (\|u^0\|_{L^2}+\|u_{\tau,N}^0\|_{L^2}) \notag \\
&\lesssim 
N^{-1} \|u^0\|_{H^1}^2 
\end{align}
and
\begin{align}\label{P-P_N-error}
|P-P_N| 
&=
\bigg| \frac{1}{2\pi} \int_{\T} (u^0 \partial_x\overline{u^0}-u^0_{\tau,N} \partial_x\overline{u}_{\tau,N}^0) \,d x \bigg| \notag \\
&\lesssim
\bigg| \frac{1}{2\pi} \int_{\T} 
\Big[ (u^0-u_{\tau,N}^0) \partial_x\overline{u^0}
+ u^0_{\tau,N} \partial_x\overline{(u^0-u_{\tau,N}^0)} \Big] \,d x \bigg| \notag \\
&=
\bigg| \frac{1}{2\pi} \int_{\T} \Big[(u^0-u_{\tau,N}^0) \partial_x\overline{u^0}
-  \partial_x u^0_{\tau,N}\overline{(u^0-u_{\tau,N}^0)} \Big]\,d x \bigg| \notag \\
&\lesssim
\|u^0-u_{\tau,N}^0\|_{L^2} (\|\partial_x u^0\|_{L^2}+\|\partial_x u^0_{\tau,N}\|_{L^2}) \notag \\
&\lesssim 
N^{-1} \|u^0\|_{H^1}^2 . 
\end{align}
From the expression of $\Phi_1^n$ we see that its Fourier coefficients can be written as
\begin{align*}
\mathcal{F}_k[\Phi_1^n]
= F(M,P;k)\hat v_k(t_n) - F(M_N,P_N;k) \hat v_k^n ,
\end{align*}
with 
\begin{align*}
F(M,P;k)
:=
\fe^{-2i\lambda\tau Pk^{-1}1_{k\ne 0} -2i\lambda\tau M} - 1+ 2i\lambda\tau Pk^{-1} 1_{k\ne 0}
+2i\lambda\tau M + (1-\fe^{-2i\lambda\tau M})1_{k=0} .
\end{align*}
By using Taylor's expansion and mean value theorem, it is straightforward to verify that 
$$
|F(M,P;k) -F(M_N,P_N;k)|
\lesssim 
\tau (|P-P_N|+|M-M_N|) . 
$$
As a result, we have 
\begin{align}\label{Phin1-Estimate}
\|\Phi_1^n \|_{H^s}
\lesssim & 
\| (\langle k \rangle^s \mathcal{F}_k[\Phi_1^n])_{k\in\Z} \|_{l^2} \notag \\
\lesssim & 
\tau(|P-P_N|+|M-M_N|) 
\| (\langle k \rangle^s \hat v_k(t_n))_{k\in\Z} \|_{l^2} 
+ \| (\langle k \rangle^s (\hat v_k(t_n)-\hat v_k))_{k\in\Z} \|_{l^2} \notag \\
\lesssim &
\tau(|P-P_N|+|M-M_N|) \|v(t_n)\|_{H^s}
+ \tau \|v(t_n)-v^n\|_{H^s} \notag \\
\lesssim &
\tau N^{-1} \|v\|_{L^\infty(0,T;H^1)}^3
+ \tau \|v(t_n)-v^n\|_{H^s} ,
\end{align}
where the last inequality follows from \eqref{M-M_N-error}--\eqref{P-P_N-error}.

Since $\Phi_2^n$ is a constant, it is straightforward to show that (similarly as \eqref{M-M_N-error}) 
\begin{align}\label{Phin2-Estimate}
|\Phi^n_2|
\lesssim &
\tau\big( \|v^n-v(t_n)\|_{L^2}(\|e^{it_n\partial_x^2}v(t_n)\|_{L^\infty}^2+\|e^{it_n\partial_x^2}v^n\|_{L^\infty}^2) \notag \\ 
\lesssim &
 \tau\big( \|v^n-v(t_n)\|_{L^2}(\|v(t_n)\|_{H^{s}}^2+\|v^n\|_{H^{s}}^2) \qquad\mbox{(this holds for $s>\frac12$)} \notag \\
 \lesssim &
 \tau\big( \|v^n-v(t_n)\|_{L^2}(\|v(t_n)\|_{H^{s}}^2+\|v^n-v(t_n)\|_{H^{s}}^2) .
\end{align} 

Similarly,  $\Phi^n_3$ can be decomposed into several functions of the following form: 
\begin{align*}
\Phi^n_3=& -2i  \sum\limits_{0\ne |k|\le N}\fe^{ikx}\sum\limits_{{\substack{k_1+k_2+k_3=k\\ |k_2+k_3|\le N}}} \bigg( \int_0^\tau\frac{k_1+k_2}{k} \fe^{i(t_n+s)\phi}\,ds \bigg) 
\hat{f}_{1,k_1}\hat{f}_{2,k_2}\hat{f}_{3, k_3},
\end{align*}
where $\hat f_{j,k}$ denotes the $k$th Fourier coefficient of the functions $f_j$, 
and one of the three functions $f_j, j=1,2,3$, is 
$$
v^n-v(t_n) \,\,\, \mbox{or its conjugate} ; 
$$
the other two of the three functions $f_j, j=1,2,3$, are either $v^n$ or $v(t_n)$ or their conjugates. We assume that $\hat f_{j,k}, k\in\Z$ are nonnegative; otherwise we consider functions with Fourier coefficients $|\hat f_{j,k}|$ as we did in the proof of Lemma \ref{lem:kato-Ponce-1}  (ii).
Then 
\begin{align*}
\big|\big(\widehat{\Phi^n_3}\big)_k\big|\lesssim &\tau \sum\limits_{k_1+k_2+k_3=k}  \frac{|k_1+k_2|}{|k|}  \>\hat{f}_{1,k_1}\hat{f}_{2,k_2}\hat{f}_{3, k_3}
=\mathcal{F}_k [ \tau J^{-1}(f_1J(f_2f_3) )] .
\end{align*}
As a result, by Plancherel's identity and Lemma \ref{lem:kato-Ponce-1} (i), we have 
\begin{align}\label{Phin3-Estimate}
\|\Phi^n_3\|_{H^{s}}
\lesssim & 
\|\tau J^{-1}(f_3J(f_1f_2) )\|_{H^{s}} \notag \\
\lesssim &\tau 
\|f_3\|_{H^{s}}\|f_1f_2\|_{H^{s}}  \qquad\mbox{(this requires $s>\frac12$)} \notag \\
\lesssim &\tau 
\|f_3\|_{H^{s}}\|f_1\|_{H^{s}}\|f_2\|_{H^{s}} \notag \\
\lesssim & 
\tau \|v^n-v(t_n)\|_{H^{s}} (\|v^n\|_{H^{s}}^2 + \|v(t_n)\|_{H^{s}}^2) \notag \\
\lesssim & 
\tau \|v^n-v(t_n)\|_{H^{s}} (\|v^n-v(t_n)\|_{H^{s}}^2 + \|v(t_n)\|_{H^{s}}^2).  
\end{align}
$\Phi_4^n$ and $\Phi_5^n$ can be estimated similarly, i.e.,
\begin{align*}
\|\Phi^n_4\|_{H^{s}} + \|\Phi^n_5\|_{H^{s}}
\lesssim & 
\tau \|v^n-v(t_n)\|_{H^{s}} (\|v^n-v(t_n)\|_{H^{s}}^2 + \|v(t_n)\|_{H^{s}}^2) .  
\end{align*}
Hence, combining with the estimates of $\Phi^n_j$, $j=1,\dots,5$, we have   
\begin{align*}
&\|\Phi^n(v(t_n);M,P)-\Phi^n(v^n;M_N,P_N)\|_{H^{s}} \\ 
&\le  (1+C\tau) \|v^n-v(t_n)\|_{H^{s}}+C\tau\|v^n-v(t_n)\|_{H^{s}}^3
+C\tau N^{-1} ,
\end{align*}
which holds for any given $s\in(\frac12,1)$. 
Substituting this and \eqref{est:LHs} into \eqref{vtn+1-vn+1} yields that 
\begin{align*}
\|v(t_{n+1})-v^{n+1}\|_{H^{s}}
   \leq & C\big(\tau^\frac32+\tau N^{-1+s}\big)+(1+C\tau) \|v^n-v(t_n)\|_{H^{s}}
        +C\tau\|v^n-v(t_n)\|_{H^{s}}^3.
\end{align*}
By using the discrete Gronwall's inequality with induction assumption on $\|v^n-v(t_n)\|_{H^{s}}\le 1$, we obtain (for sufficiently small $\tau$)
\begin{align*}
\max_{0\le n\le L}\big\|v(t_{n})-v^{n}\big\|_{H^{s}}
\lesssim \tau^\frac12+ N^{-1+s} .
\end{align*}
This proves the desired result in Lemma \ref{lem:Apriori}. 
\end{proof}

Lemma \ref{lem:Apriori} implies that $\|v(t_n)-v^n\|_{H^{s}} \lesssim 1$. Then, by using the triangle inequality and boundedness of the exact solution in $H^1$, we have  
$$
\|v^n\|_{H^s} \lesssim 
\|v(t_n)-v^n\|_{H^s} + \|v(t_n)\|_{H^s} \lesssim 1 .
$$
This result can be furthermore improved to the $H^1$ norm, as shown in the following lemma. 
\begin{lemma}
Let $u^0\in H^1(\T)$, and let $u^n_{\tau,N}$, $n=0,1,\dots,L$, be the numerical solution given by \eqref{numerical}--\eqref{psi}. Then there exists a constant $\tau_0>0$ such that for $\tau\in(0,\tau_0]$ the following estimate holds:  
\begin{align}
\max_{0\le n\le L} \|u^n_{\tau,N}\|_{H^1}
\lesssim 1.
\end{align}
\end{lemma}
\begin{proof}
Let $v^n=e^{-it_n\partial_x^2}u^n_{\tau,N}$.  
By using the expression of $\Phi^n$ in \eqref{expr-Phinf}, we immediately obtain that 
\begin{align}
\|\Phi^n(v^n;M_N,P_N)\|_{H^1}
&\le
\|v^n\|_{H^1} + C\tau \|v^n\|_{H^1} 
+ C\tau \|v^n\|_{H^1} \|v^n\|_{H^s}^2 ,
\end{align}
which holds for any fixed $s\in(\frac12,1)$. Since $\|v^n\|_{H^s}\lesssim 1$ is already proved in Lemma \ref{lem:Apriori}, substituting this into \eqref{NuSo-NLS} yields
\begin{align}
\|v^{n+1} \|_{H^1}
\le
\|v^n\|_{H^1} + C\tau \|v^n\|_{H^1} ,
\end{align}
which implies $\max\limits_{0\le n\le L} \|v^n\|_{H^1}
\lesssim 1$ after iteration in $n$. 
The desired result follows from the relation $\|v^n\|_{H^1}=\|u^n_{\tau,N}\|_{H^1}$. 
\end{proof}

\subsection{Error estimation in $L^2(\T)$}
\label{section:error-L2}

From \eqref{est:R1}, \eqref{est:R2}, \eqref{est:R3}, \eqref{est:R4}, \eqref{est:R5} and \eqref{est:R6} we conclude that 
\begin{align}\label{est:L}
\big\|\mathcal{L}^n\big\|_{L^2} &\le C\big(\tau^2\sqrt{\ln\tau^{-1}}+\tau N^{-1}\big) .
\end{align}
By choosing $s=0$ in \eqref{Phin1-Estimate} and choosing a fixed $s\in(\frac12,1)$ in \eqref{Phin2-Estimate}, we have 
\begin{align*} 
\|\Phi^n_1\|_{L^2}+\|\Phi^n_2\|_{L^2}
\lesssim & 
\tau N^{-1}+ \tau \|v^n-v(t_n)\|_{L^2} .  
\end{align*} 
Instead of \eqref{Phin3-Estimate}, we need to use the following estimate for $\Phi_3^n$:
\begin{align*} 
\|\Phi^n_3\|_{L^2}
\lesssim & 
\|\tau J^{-1}(f_3J(f_1f_2) )\|_{L^2} 
\lesssim \tau 
\min( \|f_3\|_{H^{1}}\|f_1f_2\|_{L^2} ,\|f_3\|_{L^2}\|f_1f_2\|_{H^1} ) .  
\end{align*}
which is a consequence of Lemma \ref{lem:kato-Ponce-1} (ii). 
Recall that one of the three functions $f_j, j=1,2,3$, is 
$v^n-v(t_n)$ or its conjugate, and the other two functions are either $v^n$ or $v(t_n)$ (or their conjugates). If $f_1$ is $v^n-v(t_n)$ or its conjugate, then we choose $L^2$ norm on $f_1$; otherwise we choose $L^2$ norm on $f_2f_3$. In either case we obtain 
\begin{align*}
\|\Phi^n_3\|_{L^2}
\lesssim & 
\tau \|v^n-v(t_n)\|_{L^2} (\|v(t_n)\|_{H^{1}}^2 +\|v^n\|_{H^{1}}^2 ) 
\lesssim \tau \|v^n-v(t_n)\|_{L^2}.
\end{align*}
The two terms $\Phi_4^n$ and $\Phi_5^n$ can be estimated similarly, i.e.,
\begin{align*}
\|\Phi^n_4\|_{L^2} + \|\Phi^n_5\|_{L^2}
\lesssim \tau \|v^n-v(t_n)\|_{L^2}. 
\end{align*}
Substituting the estimates of $\|\Phi^n_j\|_{L^2}$, $j=1,\dots,5$, into \eqref{Phivtn-Phivn-1}, we have 
\begin{align*}
\|\Phi^n(v(t_n);M,P)-\Phi^n(v^{n};M_N,P_N)\|_{L^2}
\lesssim \tau N^{-1}+ \tau \|v^n-v(t_n)\|_{L^2} .
\end{align*}
Then, substituting this into \eqref{vtn+1-vn+1} and using estimate \eqref{est:L}, we obtain 
\begin{align}
\|v(t_{n+1})-v^{n+1}\|_{L^2}
\le C\big(\tau^2\sqrt{\ln\tau^{-1}}+\tau N^{-1}\big)+(1+C\tau) \|v^n-v(t_n)\|_{L^2}. 
\end{align}
Iterating this inequality yields 
\begin{align*}
\max_{1\le n\le L} \|v(t_{n})-v^{n}\|_{L^2}
\lesssim \|v(t_{0})-v^{0}\|_{L^2}
+ \tau \sqrt{\ln\tau^{-1}}+ N^{-1} 
\lesssim \tau \sqrt{\ln\tau^{-1}}+ N^{-1} .
\end{align*}
This completes the proof of Theorem \ref{main:thm1} in view of $ \|v(t_{n})-v^{n}\|_{L^2}= \|u(t_{n})-u^{n}_{\tau,N}\|_{L^2}$.
\qed

\section{Numerical experiments} \label{sec:numerical}

In this section we present numerical experiments to support the theoretical analysis presented in Theorem \ref{main:thm1}. 
We consider the NLS equation \eqref{model} with $\lambda=-1$ and initial value 
\begin{align}\label{initial-u0}
u^0(x) = \frac{1}{10} \sum_{0\ne k\in\Z} |k|^{-0.51-\alpha} e^{ikx} ,
\end{align}
which satisfies that $u^0\in H^{\alpha}(\T)$ and $u^0\notin H^{\alpha-0.01}(\T)$. 

We solve the problem by the proposed method \eqref{numerical}--\eqref{psi} for $\alpha=2$ and $\alpha=1$, respectively, and present the time discretisation errors $\|u_{\tau,N} - u_{\tau/2,N}\|_{L^2}$ in Tables \ref{Table1}--\ref{Table2} for several sufficiently large $N$. From the numerical results we can see that the error from spatial discretisation is negligibly small in observing the temporal convergence rates, i.e., almost first-order convergent as $\tau\rightarrow 0$. This is consistent with the theoretical result proved in Theorem \ref{main:thm1}. 

\begin{table}[htp]\centering\small
\caption{\small Temporal discretisation error $\|u_{\tau,N} - u_{\tau/2,N}\|_{L^2}$ at $T=1$\\ 
\indent\hspace{48pt}
with $\alpha=2$ in \eqref{initial-u0} (for $H^2$ initial data). }\vspace{-5pt}
\centering
\setlength{\tabcolsep}{4.0mm}{
\begin{tabular}{cccccc}
\toprule
         &&         $N=2^{8}$&         $N=2^{9}$&         $N=2^{10}$&      \\[2pt]
\midrule
$\tau= 2^{-6}$           &&      7.662E--06&     7.662E--06&      7.662E--06&   \\[2pt]
$\tau= 2^{-7}$           &&      3.829E--06&     3.829E--06&      3.829E--06&     \\[2pt]
$\tau= 2^{-8}$   &&              1.915E--06&     1.915E--06&      1.915E--06&    \\[2pt]
\midrule
convergence rate && $O(\tau^{1.00})$& $O(\tau^{1.00})$& $O(\tau^{1.00})$  \\[2pt]
\bottomrule
\end{tabular}}
\label{Table1}
\vspace{20pt} 
\caption{\small Temporal discretisation error $\|u_{\tau,N} - u_{\tau/2,N}\|_{L^2}$ at $T=1$\\ 
\indent\hspace{48pt}
with $\alpha=1$ in \eqref{initial-u0} (for $H^1$ initial data). }\vspace{-5pt}
\centering
\setlength{\tabcolsep}{4.0mm}{
\begin{tabular}{cccccc}
\toprule
        &&         $N=2^{8}$&         $N=2^{9}$&         $N=2^{10}$&      \\[2pt]
\midrule
$\tau= 2^{-6}$           &&     2.144E--05&    2.146E--05&       2.146E--05&     \\[2pt]
$\tau= 2^{-7}$           &&     1.023E--05&    1.021E--05&       1.021E--05&     \\[2pt]
$\tau= 2^{-8}$           &&     5.057E--06&    5.067E--06&       5.066E--06&     \\[2pt]
\midrule
convergence rate && $O(\tau^{1.02})$& $O(\tau^{1.02})$& $O(\tau^{1.02})$  \\[2pt]
\bottomrule
\end{tabular}}
\label{Table2}
\end{table}

We present the spatial discretisation errors $\|u_{\tau,N} - u_{\tau,2N}\|_{L^2}$ for $\alpha=2$ and $\alpha=1$ in Tables \ref{Table3}--\ref{Table4} for several sufficiently small stepsize $\tau$. 
From the numerical results we can see that the error from temporal discretisation is negligibly small in observing the spatial convergence rates, i.e., $\alpha$th-order convergence for $H^\alpha$ initial data. This is consistent with the result proved in Theorem \ref{main:thm1} and the comments in Remark \ref{Remark}. 

\begin{table}[htp]\centering\small
\vspace{5pt}
\caption{\small Spatial discretisation error $\|u_{\tau,N} - u_{\tau,2N}\|_{L^2}$ at $T=1$\\ 
\indent\hspace{48pt}
with $\alpha=2$ in \eqref{initial-u0} (for $H^2$ initial data). }\vspace{-5pt}
\centering
\setlength{\tabcolsep}{4.0mm}{
\begin{tabular}{cccccc}
\toprule
           &&         $\tau = 2^{-8}$&         $\tau = 2^{-9}$&         $\tau = 2^{-10}$&      \\[2pt]
\midrule
$N=16$           &&      2.430E--04&     2.430E--03&     2.430E--03&   \\[2pt]
$N=32$           &&      6.237E--05&     6.237E--05&     6.237E--05&    \\[2pt]
$N=64$           &&      1.574E--05&     1.574E--05&     1.574E--05&    \\[2pt]
\midrule
convergence rate && $O(N^{-1.99})$& $O(N^{-1.99})$& $O(N^{-1.99})$  \\[2pt]
\bottomrule
\end{tabular}}
\label{Table3}
\vspace{20pt}
\caption{\small Spatial discretisation error $\|u_{\tau,N} - u_{\tau,2N}\|_{L^2}$ at $T=1$\\ 
\indent\hspace{48pt}
with $\alpha=1$ in \eqref{initial-u0} (for $H^1$ initial data). }\vspace{-5pt}
\centering
\setlength{\tabcolsep}{4.0mm}{
\begin{tabular}{cccccc}
\toprule
           &&         $\tau = 2^{-8}$&         $\tau = 2^{-9}$&         $\tau = 2^{-10}$&      \\[2pt]
\midrule
$N=16$           &&      5.056E--03&     5.056E--03&      5.056E--03&   \\[2pt]
$N=32$           &&      2.559E--03&     2.559E--03&      2.559E--03&     \\[2pt]
$N=64$           &&      1.283E--03&     1.283E--03&      1.283E--03&    \\[2pt]
\midrule
convergence rate && $O(N^{-1.00})$& $O(N^{-1.00})$& $O(N^{-1.00})$  \\[2pt]
\bottomrule
\end{tabular}}
\label{Table4}
\end{table}

\section{Conclusion} \label{sec:conclusion}
 
We have constructed a fast fully discrete low-regularity integrator for solving the NLS equation with nonsmooth initial data in one dimension. 
The method can be implemented by using FFT with $O(N\ln N)$ operations at every time level, and is proved to have an error bound of $O(\tau\sqrt{\ln(1/\tau)}+N^{-1})$ when the initial data is in $H^1(\T)$. 
For initial data in $H^s(\T)$ with $s>1$, the numerical results show that the proposed method can have an error bound of $O(\tau+N^{-s})$. 
We expect that the techniques for constructing and analysing the spatial discretisation method in combination with the temporal low-regularity integrator may also be extended to other dispersive equations with nonsmooth data.

\vskip 25pt

\bibliographystyle{model1-num-names}

\begin{thebibliography}{00}

\bibitem{besse}
C. Besse, B. Bid\'{e}garay, and S. Descombes: 
Order estimates in time of splitting methods
for the nonlinear Schr\"odinger equation. 
{\it SIAM J. Numer. Anal.} 40 (2002), pp. 26--40.

\bibitem{Bo}
{J. Bourgain}: 
Fourier transform restriction phenomena for certain lattice subsets and
applications to nonlinear evolution equations. I. Schr\"odinger  equations. 
{\it Geom. Funct. Anal.} 3 (1993), pp. 107--156.
 
%
\bibitem{BoLi-KatoPonce}
{J. Bourgain and D. Li}: 
On an endpoint Kato-Ponce inequality. 
{\it Differential Integral Equations} 27 (2014), pp. 1037--1072.

\bibitem{Ca-book-03}
T. Cazenave:
{\it Semilinear Schr\"odinger equations.} 
Courant Lecture Notes in Mathematics 10, American Mathematical Society, 2003.

\bibitem{Chu-2008}
E. Chu: 
{\it Discrete and continuous Fourier transforms analysis, applications and fast algorithms.} CRC Press, New York, 2008.  

\bibitem{ESS-2016}
J. Eilinghoff, R. Schnaubelt, and K. Schratz: 
Fractional error estimates of splitting schemes for the nonlinear Schrödinger equation. 
{\it J. Math. Anal. Appl.} 442 (2016), pp. 740--760.

\bibitem{Rousset-Schratz-2020}
F. Rousset and K. Schratz: A general framework of low regularity integrators. 
arXiv:2010.01640

\bibitem{Hochbruck-Ostermann-2010}
M. Hochbruck and A. Ostermann: 
Exponential integrators. 
{\it Acta Numerica} 19 (2010), pp. 209--286.

\bibitem{Hofmanova-Schratz-2017}
M. Hofmanov\'aa and K. Schratz: 
An exponential-type integrator for the KdV equation. 
{\it Numer. Math.} 136 (2017), pp. 1117--1137.

\bibitem{Ignat-2011}
L. I. Ignat: A splitting method for the nonlinear Schrödinger equation.
{\it J. Differential Equations} 250 (2011), pp. 3022--3046.

\bibitem{Knoller-Ostermann-Schratz-2019}
M. Kn\"oller, A. Ostermann, and K. Schratz: 
A Fourier integrator for the cubic nonlinear Schr\"odinger equation with
rough initial data. 
{\it SIAM J. Numer. Anal.} 57 (2019), pp. 1967--1986.

\bibitem{Kato-Ponce}
{T. Kato and G. Ponce}: 
Commutator estimates and the Euler and Navier-Stokes equations. 
{\it Commun. Pure Appl. Math.} 41 (1988) pp. 891-907.

\bibitem{Li-KatoPonce}
{D. Li}: On Kato-Ponce and fractional Leibniz. 
{\it Rev. Mat. Iberoam.} 35 (2019) pp. 23--100.

\bibitem{Lubich-2008}
Ch. Lubich: 
On splitting methods for Schr\"odinger-Poisson and cubic nonlinear Schr\"odinger equations. {\it Math. Comp.} 77 (2008), pp. 2141--2153.

\bibitem{Ostermann-Schratz-FoCM2}
A. Ostermann, F. Rousset, and K. Schratz: 
Error estimates of a Fourier integrator for the cubic Schr\"odinger equation
at low regularity. {\it Found. Comput. Math.} (2020), DOI: 10.1007/s10208-020-09468-7

\bibitem{Ostermann-Rousset-Schratz-JEMS}
A. Ostermann, F. Rousset, and K. Schratz: 
Fourier integrator for periodic NLS: low regularity estimates via discrete Bourgain spaces. 
arXiv:2006.12785

\bibitem{Ostermann-Schratz-FoCM}
A. Ostermann and K. Schratz: 
Low regularity exponential-type integrators for semilinear Schr\"odinger equations. 
{\it Found. Comput. Math.} 18 (2018), pp. 731--755.


\bibitem{ostermann-su}
{A. Ostermann and C. Su}: 
 A Lawson-type exponential integrator for the Korteweg-de Vries equation. 
IMA J. Numer. Anal. (2020), DOI: 10.1093/imanum/drz030


\bibitem{Sanz-Serna1984}
J. M. Sanz-Serna: 
Methods for the numerical solution of the nonlinear Schr{\"o}dinger equation. 
{\it Math. Comp.} 43 (1984), pp. 21--27. 

\bibitem{diraclow}
{K. Schratz, Y. Wang, and X. Zhao}:
Low-regularity integrators for nonlinear Dirac equations. 
to appear in {\it Math. Comp.} (2020), arXiv:1906.09413

\bibitem{Wang2014}
J. Wang: 
A new error analysis of {C}rank--{N}icolson {G}alerkin {FEM}s for a generalized nonlinear {S}chr{\"o}dinger equation. 
{\it J. Sci. Comput.} 60 (2014), pp. 390--407.

\bibitem{Wu-Yao-2020}
Y. Wu and F. Yao: A first-order Fourier integrator for the nonlinear Schr\"odinger equation on $\T$ without loss of regularity. 
Preprint, arXiv:2010.02672

\bibitem{WuZhao-1} 
{Y. Wu and  X. Zhao}: 
Optimal convergence of a first order low-regularity integrator for the KdV equation. arXiv:1910.07367, 2019.

\bibitem{wu}
{Y. Wu and  X. Zhao}: 
Embedded exponential-type low-regularity integrators for KdV equation under rough data. 
arXiv:2008.07053, 2020.

\bibitem{DFT-wiki}
Wikipedia: \url{https://en.wikipedia.org/wiki/Discrete_Fourier_transform}


\end{thebibliography}

\end{document}